\documentclass[10pt]{amsart}

\usepackage{amsmath}
\usepackage{amsfonts}
\usepackage{amsthm}
\usepackage{amssymb}
\usepackage[shortlabels]{enumitem}
\usepackage{tikz-cd}
\usepackage{graphicx}
\usepackage{footmisc}
\usepackage{centernot}
\usepackage{mathtools}
\usepackage{stmaryrd}
\usepackage{url}
\usepackage{multicol}
\usepackage{pdfpages}
\usepackage{colortbl}
\usepackage{multirow}
\usepackage{hhline}

\usepackage[style=numeric]{biblatex}
\newtheorem{definition}{Definition}[section]
\newtheorem{theorem}[definition]{Theorem}
\newtheorem{lemma}[definition]{Lemma}
\newtheorem{corollary}[definition]{Corollary}
\newtheorem{proposition}[definition]{Proposition}
\newtheorem{conjecture}[definition]{Conjecture}
\newtheorem*{thmrepeat}{Theorem 1.3}

\makeatletter
\newcommand{\newreptheorem}[2]{\newtheorem*{rep@#1}{\rep@title}\newenvironment{rep#1}[1]{\def\rep@title{#2 \ref*{##1}}\begin{rep@#1}}{\end{rep@#1}}}
\makeatother

\theoremstyle{definition}

\newtheorem{remark}[definition]{Remark}

\newcommand{\gal}{\mathrm{Gal}}

\newcommand{\rank}{\mathrm{rank}}

\newcommand{\im}{\mathrm{im}}

    \DeclareFontFamily{U}{wncy}{}
    \DeclareFontShape{U}{wncy}{m}{n}{<->wncyr10}{}
    \DeclareSymbolFont{mcy}{U}{wncy}{m}{n}
    \DeclareMathSymbol{\Sh}{\mathord}{mcy}{"58}
\title{Imaginary quadratic fields $F$ with $X_0(15)(F)$ finite.}

\author[T.~Evink]{Tim Evink}
\address{Institute of Algebra and Number Theory, Ulm University, Helmholtzstr.~18, 89081 Ulm, Germany.}
\email{tim.evink@uni-ulm.de}

\begin{document}

\begin{abstract}
    Caraiani and Newton have proven that if $F$ is an imaginary quadratic number field such that $X_0(15)$ has rank $0$ over $F$, then every elliptic curve over $F$ is modular. This paper is concerned with the quadratic fields $F=\mathbb{Q}(\sqrt{-p})$ for a prime number $p$. We give explicit conditions on $p$ under which the rank is $0$, and prove that these conditions are satisfied for $87,5\%$ of the primes for which the rank is expected to be even based on the parity conjecture. We also show these conditions are satisfied if and only if rank $0$ follows from a $4$-descent over $\mathbb{Q}$ on the quadratic twist $X_0(15)_{-p}$. To prove this, we perform two consecutive $2$-descents and prove this gives rank bounds equivalent to those obtained from a $4$-descent using visualisation techniques for $\Sh[2]$. In fact we prove a more general connection between higher descents for elliptic curves which seems interesting in its own right.
\end{abstract}

\maketitle

\section{Introduction}
Consider the modular curve $X_0(15)$, which is an elliptic curve over $\mathbb{Q}$ of rank $0$. Caraiani and Newton \cite{caraianinewton2023modularity} have proven that if $F$ is an imaginary quadratic number field such that $X_0(15)$ has rank $0$ over $F$, then every elliptic curve over $F$ is modular. For positive square-free integers $d$, the global root number of $X_0(15)$ over $\mathbb{Q}(\sqrt{-d})$ turns out to be $1$ precisely when
\begin{equation}\label{eq:rootnumbercong}
    d\equiv 0,1,2,3,4,5,8,12\bmod 15.
\end{equation}
Thus, based on the parity conjecture we expect the rank of $X_0(15)$ over $\mathbb{Q}(\sqrt{-d})$ to be even for precisely those $d$ satisfying \eqref{eq:rootnumbercong}. Forthcoming work of Smith will show that we have rank $0$ for a proportion of 100\% of such $d$, see the introduction of \cite{caraianinewton2023modularity}. However, this density result might give a misleading picture for the amount of occurrences of positive rank up to some bound: out of the $3158$ positive square-free integers $d\leq 10^4$ satisfying \eqref{eq:rootnumbercong}, we have $514$ instances of positive rank\footnote{In all these instances of positive rank, the rank is in fact $2$. The smallest positive square-free $d$ such that $\rank(X_0(15)/\mathbb{Q}(\sqrt{-d}))=4$ is $d=20445$. These conclusions were found by calling the Magma \cite{magma} function \texttt{Rank} on the quadratic twist $X_0(15)_{-d}$.}.\\
In this paper we focus on the case where $d=p$ is a prime number, and give explicit conditions on $p$ for which rank $0$ occurs. The main result concerning occurrences of rank $0$ is formulated using R\'edei symbols $[a,b,c]$ as developed in \cite{stevenhagen2021redei}, a tri-linear number theoretic symbol taking values in $\{\pm 1\}$ that encodes splitting behaviour of rational primes in certain $D_4$ or $C_4$-extensions of $\mathbb{Q}$.
\begin{theorem}\label{thm:mainrank}
Let $p$ be a prime satisfying one of the following conditions.
\begin{enumerate}[(i)]
\item $p\equiv 1,2,3,4,5,8\bmod 15$ and $p\not\equiv 1\bmod 8$, 
\item $p\equiv 1,4\bmod 15$, $p\equiv 1\bmod 8$ and $[-1,10,p]=-1$,
\item $p\equiv 2,8\bmod 15$, $p\equiv 1\bmod 8$ and $[-1,2,p]=1$.
\end{enumerate}
Then $X_0(15)$ has rank $0$ over $\mathbb{Q}(\sqrt{-p})$. Moreover, the density of these primes amongst the primes $p$ of the form $p\equiv 1,2,3,4,5,8\bmod 15$ is $87,5\%$.
\end{theorem}
The stated R\'edei symbols in Theorem \ref{thm:mainrank} can be computed using the following.
\begin{enumerate}[(i)]
    \item If $p\equiv 1,9\bmod 40$, then $[-1,10,p]=1\,\Leftrightarrow\, p$ splits completely in $\mathbb{Q}(\sqrt{3+\sqrt{-1}})$.
    \item If $p\equiv 1\bmod 8$, then $[-1,2,p]=1\,\Leftrightarrow\, p$ splits completely in $\mathbb{Q}(\sqrt{1+\sqrt{-1}})$.
\end{enumerate}
See \cite[\S 7]{stevenhagen2021redei}. Combining Theorem \ref{thm:mainrank} with \cite[Thm. 1.1]{caraianinewton2023modularity} we immediately obtain the following corollary.
\begin{corollary}
    Let $p$ be a prime satisfying one of conditions (i)-(iii) of Theorem \ref{thm:mainrank}. Then every elliptic curve defined over $\mathbb{Q}(\sqrt{-p})$ is modular.
\end{corollary}
Let us illustrate the manner in which Theorem \ref{thm:mainrank} is proven. For a prime number $p$ we will see that
\[
\rank(X_0(15)/\mathbb{Q}(\sqrt{-p}))=\rank(X_0(15)_{-p}/\mathbb{Q}),
\]
which makes the quadratic twist $X_0(15)_{-p}$ a natural object of study to determine the rank of $X_0(15)$ over $\mathbb{Q}(\sqrt{-p})$. When $p\equiv 1,2,3,4,5,8\bmod 15$, a basic $2$-Selmer group computation of $X_0(15)_{-p}$ over $\mathbb{Q}$ will show that the rank is $0$ when $p\not\equiv 1\bmod 8$, and at most $2$ when $p\equiv 1\bmod 8$. In the latter case we improve on the upper bound of $2$ by computing $2$-Selmer groups over suitable quadratic number fields in a similar fashion to \cite{evinkheidentop2021} and \cite{evinktoptop2022}, leading to a proof of Theorem \ref{thm:mainrank}.\\
We also analyse how this improvement relates to a natural way of improving on a $2$-descent, namely higher descents. We will show that Theorem \ref{thm:mainrank} implies rank $0$ precisely when this follows from a $4$-descent on $X_0(15)_{-p}$ over $\mathbb{Q}$. In proving this connection, the following result is instrumental.
\begin{theorem}\label{thm:mainsha}
     Let $E$ be an elliptic curve defined over a number field $K$. Let $L=K(\sqrt{d})$ be a quadratic extension such that the following conditions are satisfied.
     \begin{enumerate}
         \item We have $\Sh(E_{d}/K)[2]=0$.
         \item The corestriction map $\Sh(E/L)[2]\to\Sh(E/K)[2]$ is injective.
     \end{enumerate}
     Then the restriction map $\Sh(E/K)\to\Sh(E/L)$ induces an isomorphism
     \begin{equation}\label{eq:shaiso}
     \Sh(E/K)[2^{\infty}]/\Sh(E/K)[2]\xrightarrow{\sim}\Sh(E/L)[2^{\infty}].
     \end{equation}
\end{theorem}
Theorem \ref{thm:mainsha} is proven in the setting of visualisation of $\Sh[2]$ as in \cite[\S 4]{Bruin_2004}. In loc. cit. an explicit necessary condition for $2$-divisibility in $\Sh[2]$ is given, which can be used to improve on a $2$-descent. The novelty of Theorem \ref{thm:mainsha} lies in the fact that it gives conditions under which the necessary condition is also sufficient, thus connecting a $2^{n+1}$-descent of $E/K$ to a $2^n$-descent of $E/L$. For example, elements of order $2$ in $\Sh(E/L)$ give rise to elements of order $4$ in $\Sh(E/K)$, a conclusion that doesn't follow from the results in \cite{Bruin_2004}. The main idea is to generalise \cite[Lemma 7.1]{BruinHem} to the setting of a $2$-isogeny between Abelian surfaces as in \cite[\S 4]{Bruin_2004}, which allows control over the image of $\Sh(E/K)\to \Sh(A/K)$ as in loc. cit.\\
Using Theorem \ref{thm:mainsha} we show that the rank results of Theorem \ref{thm:mainrank} are equivalent to those obtained from performing a $4$-descent on the twist $X_0(15)_{-p}$ over $\mathbb{Q}$. Specifically, using the Chinese remainder theorem to combine the congruence conditions mod $15$ and mod $8$ from Theorem \ref{thm:mainrank} $(ii)$ and $(iii)$, we prove the following results.
\begin{theorem}\label{thm:resp}
Let $p\equiv 1,49\bmod 120$ be a prime and let $E=X_0(15)$. 
\begin{enumerate}[(i)]
    \item\label{thm:resp:i} If $[-1,10,p]=-1$, then $\rank(E_{-p}/\mathbb{Q})=0$ and $\Sh(E_{-p}/\mathbb{Q})[4]\cong (\mathbb{Z}/2\mathbb{Z})^2$.
    \item\label{thm:resp:ii} If $[-1,10,p]=1$, then $\Sh(E_{-p}/\mathbb{Q})[4]\xrightarrow{\cdot 2}\Sh(E_{-p}/\mathbb{Q})[2]$ is surjective.
    \item\label{thm:resp:iii} If $[-1,10,p]=1$ and $\rank(E_{-p}/\mathbb{Q})=0$, then $\Sh(E_{-p}/\mathbb{Q})[4]\cong (\mathbb{Z}/4\mathbb{Z})^2$.
\end{enumerate}
\end{theorem}
\begin{theorem}\label{thm:respq}
Let $p\equiv 17,113\bmod 120$ be a prime and let $E=X_0(15)_{-p}$. 
\begin{enumerate}[(i)]
    \item\label{thm:respq:i} If $[-1,2,p]=1$, then $\rank(E_{-p}/\mathbb{Q})=0$ and $\Sh(E_{-p}/\mathbb{Q})[4]\cong (\mathbb{Z}/2\mathbb{Z})^2$.
    \item\label{thm:respq:ii} If $[-1,2,p]=-1$, then $\Sh(E_{-p}/\mathbb{Q})[4]\xrightarrow{\cdot 2}\Sh(E_{-p}/\mathbb{Q})[2]$ is surjective.
    \item\label{thm:respq:iii} If $[-1,2,p]=-1$ and $\rank(E_{-p}/\mathbb{Q})=0$, then $\Sh(E_{-p}/\mathbb{Q})[4]\cong (\mathbb{Z}/4\mathbb{Z})^2$.
\end{enumerate}
\end{theorem}
The structure of the paper is as follows. In section \ref{sec:2descentsection} we discuss the explicit method of $2$-descent we will use. Section \ref{sec:2selmergroupsection} covers all the $2$-Selmer group computations and provides a proof of Theorem \ref{thm:mainrank}. In section \ref{sec:shasection} we cover the visualisation arguments that enable us to prove Theorem \ref{thm:mainsha}, which combined with the $2$-Selmer group computations of section \ref{sec:2selmergroupsection} allows for a proof of Theorems \ref{thm:resp} and \ref{thm:respq}. Lastly, section \ref{sec:conclusionsection} has some closing remarks about higher descents, a discussion of the applicability of Theorem \ref{thm:mainsha} in general, and the rank of $X_0(15)$ over other families of imaginary quadratic fields like $\mathbb{Q}(\sqrt{-2p})$.

\section{2-descent}\label{sec:2descentsection}
Let $K$ be a number field and $A$ and Abelian variety over $K$. The Mordell-Weil theorem states that $A(K)$ is a finitely generated Abelian group, and the method of $2$-descent provides an upper bound on its rank $A(K)$. Especially in the case of elliptic curves, this is a well-known procedure, see for example \cite[Ch. X]{silverman2000}. The main reference of this section is \cite{schaefer1995}.\\

First some notation. The Galois cohomology groups $H^*(G_K,M)$, for $G_K$ the absolute Galois group of $K$ will be denoted by $H^*(K,M)$, and when $M=G(\overline{K})$ for a commutative group scheme $G$ over $K$ we will further abbreviate the notation to $H^*(K,G)$. For example if $E$ is an elliptic curve over $K$ we will write $H^1(K,E)$ and $H^1(K,E[2])$ for $H^1(G_K,E(\overline{K}))$ and $H^1(G_K,E(\overline{K})[2])$, respectively.\\

The short exact sequence $0\to A[2]\to A\xrightarrow{2}A\to 0$ gives through Galois cohomology the long exact sequence
\begin{equation}\label{eq:basiclongexactseq}
    0\to A(K)[2]\to A(K)\xrightarrow{2} A(K)\xrightarrow{\delta} H^1(K,A[2])\to H^1(K,A)\to H^1(K,A)\to\dotsb
\end{equation}
For a completion $K_{\mathfrak{p}}$, we can identify $G_{K_{\mathfrak{p}}}$ with a decomposition group $D_{\mathfrak{p}}\subset G_K$ over $\mathfrak{p}$, hence we have restriction maps $\mathrm{res}_{\mathfrak{p}}$ on cohomology\footnote{These do not depend on the choice of decomposition group as different choices differ by a conjugation map, which is trivial on cohomology.},
which form compatible maps with the long exact sequence \eqref{eq:basiclongexactseq} and its local version involving $A(K_{\mathfrak{p}})\xrightarrow{\delta_{\mathfrak{p}}}H^1(K_{\mathfrak{p}},A[2])$. We define
\[
\begin{split}
S^2(A/K)&:=\ker\left(H^1(K,A[2])\to\prod_{\mathfrak{p}}H^1(K_{\mathfrak{p}},A)\right),\\
\Sh(A/K)&:=\ker\left(H^1(K,A)\to\prod_{\mathfrak{p}}H^1(K_{\mathfrak{p}},A)\right).
\end{split}
\]
where $S^2(A/K)$ is the $2$-Selmer group of $A$, and $\Sh(A/K)$ the Tate-Shafarevich group of $A$,
which fit in the short exact sequence
\begin{equation}\label{eq:selexact}
0\to A(K)/2A(K)\to S^2(A/K)\to \Sh(A/K)[2]\to 0.
\end{equation}
Let $S$ be the set consisting of the primes of $K$ that are either infinite, lying over $2$, or are primes of bad reduction. Denote by
$H^1(K,A[2];S)$ the subgroup of $H^1(K,A[2])$ consisting of the cocycles unramified outside $S$. Then $H^1(K,A[2];S)$ is finite and contains $S^{2}(A/K)$. In fact
\begin{equation}\label{eq:selmerfinite}
    S^{2}(A/K)=\{\xi\in H^1(K,A[2];S):\mathrm{res}_{\mathfrak{p}}(\xi)\in\im(\delta_{\mathfrak{p}})\text{ for all }\mathfrak{p}\in S\},
\end{equation}
see \cite[Ch. I, Cor. 6.6]{milne2006}. In particular, since $A(K)/2A(K)$ has $\mathbb{F}_2$-dimension $\mathrm{rank}(A/K)+A(K)[2]$, we obtain from \eqref{eq:selexact} the identity
\begin{equation}\label{eq:selmerranksha}
\mathrm{rank}(A/K)+\dim_{\mathbb{F}_2}\Sh(A/K)[2]=\dim_{\mathbb{F}_2}S^{2}(A/K)-\dim_{\mathbb{F}_2}A(K)[2].
\end{equation}
The quantity $\dim_{\mathbb{F}_2}S^{2}(A/K)-\dim_{\mathbb{F}_2}A(K)[2]$, sometimes called the $2$-Selmer rank, provides an upper bound on $\mathrm{rank}(A/K)$, and the obstruction to its sharpness is measured by $\dim_{\mathbb{F}_2}\Sh(A/K)[2]$.\\

Now assume that $A=E$ is an elliptic curve given by the equation $y^2=f(x)$, for $f(x)\in K[x]$ monic of degree $3$. In order to compute the $2$-Selmer group $S^{2}(E/K)$, we make the cohomological setup more explicit. Consider the \'etale $K$-algebra $M_K=K[t]/f(t)$. The norm map $M_K\to K$ induces a map $M_K^*/M_K^{*2}\to K^{*}/K^{*2}$, of which we denote the kernel by $M_K'$.
\[
M_K':=\ker(M_K^*/M_K^{*2}\to K^{*}/K^{*2}),
\]
Then \cite[Thm. 1.1 \& 1.2]{schaefer1995} establishes an isomorphism $H^1(K,E[2])\cong M_K'$, under which the connecting homomorphism $E(K)\to H^1(K,E[2])$ translates to the map $E(K)\to M_K'$ determined by
\begin{equation}\label{eq:explicitkummer1}
(a,b)\mapsto a-t,\quad b\neq 0.
\end{equation}
This isomorphism $H^1(K,E[2])\cong M_K'$ also holds upon replacing $K$ with a completion $K_{\mathfrak{p}}$, and upon applying both isomorphisms, the restriction map $H^1(K,E[2])\to H^1(K_{\mathfrak{p}},E[2])$ corresponds to the natural map $M_K'\to M_{K_{\mathfrak{p}}}'$ induced by $K\to K_{\mathfrak{p}}$.\\
We will also rewrite this embedding according to how $M_K$ splits as a product of fields $K_i$, thus representing elements of $H^1(K,E[2])$ as tuples $(x_i)$, with $x_i\in K_i^{*}/K_i^{*2}$. Under this correspondence, elements in $H^1(K,E[2];S)$ correspond to those tuples $(x_i)$ for which $K_i(\sqrt{x_i})$ is unramified outside $S$ for all $i$.\\

Now suppose that $E(\overline{K})[2]=E(K)[2]$, that is $f(x)$ splits completely in $K[x]$, say as $f(x)=(x-e_1)(x-e_2)(x-e_3)$. Then $M_K\cong K\times K\times K$ via $t\mapsto (e_1,e_2,e_3)$, which induces an isomorphism $M_K'\cong \ker\left(\bigoplus_{i=1}^3K^{*}/K^{*2}\xrightarrow{N} K^{*}/K^{*2}\right)$. Here $N(x_1,x_2,x_3)=x_1x_2x_3$.
Also let $d\in K^{*}$ and consider the quadratic twist $E_{d}/K$ with equation $dy^2=(x-e_1)(x-e_2)(x-e_3)$. If $(x',y')=(dx,d^2y)$ we obtain the Weierstrass equation $(y')^2=(x'-de_1)(x'-de_3)(x'-de_3)$, of which the corresponding \'etale algebra is isomorphic to $M_K$, so we can translate the embedding of $E_{d}$ using the model $dy^2=(x-e_1)(x-e_2)(x-e_3)$. It is given as follows.
\begin{align}\label{eq:kummerexplicit}
\begin{split}
E_{d}(K)/2E_{d}(K)&\xrightarrow{\delta}\ker\left(\bigoplus_{i=1}^3K^{*}/K^{*2}\to K^{*}/K^{*2}\right),\\
(a,b)&\mapsto d(a-e_1,a-e_2,a-e_3),\quad\text{ if }b\neq 0\\
(e_1,0)&\mapsto (*,d(e_1-e_2),d(e_1-e_3)),\\
(e_2,0)&\mapsto (d(e_2-e_1),*,d(e_2-e_3)),\\
(e_3,0)&\mapsto (d(e_3-e_1),d(e_3-e_2),*).
\end{split}
\end{align}
The $*$-values are determined by other coordinates, for example the first coordinate of the image of $(e_1,0)$ is $(e_1-e_2)(e_1-e_3)$.
\begin{remark}\label{rem:squareremark}
    If $d\in K^*$ is a square, the explicit Kummer-embedding of $E_{d}$ from \eqref{eq:kummerexplicit} is identical to the one of $E$. In particular, the abstract isomorphism of Selmer groups of $E$ and $E_{d}$ induced by the obvious isomorphism $E\cong E_{d}$ becomes an actual equality in the explicit setup.
\end{remark}
We then have the following commutative diagram
\begin{equation}\label{eq:explicitkummer}
\begin{tikzcd}
E_{d}(K)/2E_{d}(K) \arrow[r, "\delta"] \arrow[d]                                    & \ker\left(\displaystyle\bigoplus_{i=1}^3K^*/K^{*2}\to K^*/K^{*2}\right) \arrow[d, "\iota_{\mathfrak{p}}"]                                                    \\
E_{d}(K_{\mathfrak{p}})/2E_{d}(K_{\mathfrak{p}}) \arrow[r, "\delta_{\mathfrak{p}}"] & \ker\left(\displaystyle\bigoplus_{i=1}^3K_{\mathfrak{p}}^*/K_{\mathfrak{p}}^{*2}\to K_{\mathfrak{p}}^*/K_{\mathfrak{p}}^{*2}\right),
\end{tikzcd}
\end{equation}
And we interpret the $2$-Selmer group $S^2(E/K)$ as a subgroup of the top right kernel in \eqref{eq:explicitkummer}.
For a finite set $S$ of primes of $K$, let $K(S)=\{x\in K^{*}/K^{*2}:v_{\mathfrak{p}}(x)\equiv 0\bmod 2\text{ for all finite }\mathfrak{p}\notin S\}$. Then \eqref{eq:selmerfinite} translates to
\[
S^{2}(E/K)=\left\{x\in\ker\left(\bigoplus_{i=1}^3K(S)\to K(S)\right):\iota_{\mathfrak{p}}(x)\in\mathrm{im}(\delta_{\mathfrak{p}})\text{ for all }\mathfrak{p}\in S\right\},
\]
where $S$ is the set consisting of finite primes of $K$ where $E$ has bad reduction, the finite primes over $2$ and the real primes of $K$, and $K(S)=\{x\in K^{*}/K^{*2}:v_{\mathfrak{p}}(x)\equiv 0\bmod 2\text{ for all finite }\mathfrak{p}\notin S\}$. For $x=(e_1,e_2,e_3)\in\ker\left(\bigoplus_{i=1}^3K^{*}/K^{*2}\to K^{*}/K^{*2}\right)$ we have $\iota_{\mathfrak{p}}(x)\in\mathrm{im}(\delta_{\mathfrak{p}})$ if and only if $v_{\mathfrak{p}}(e_i)\equiv 0\bmod 2$ for all $i$. Thus if we let $S$ be the set consisting of finite primes of $K$ where $E$ has bad reduction, the finite primes over $2$ and the real primes of $K$, we have
\[
S^{2}(E/K)=\left\{x\in\ker\left(\bigoplus_{i=1}^3K(S)\to K(S)\right):\iota_{\mathfrak{p}}(x)\in\mathrm{im}(\delta_{\mathfrak{p}})\text{ for all }\mathfrak{p}\in S\right\}.
\]
Note that local images at complex primes are trivial so we can omit those in $S$. The set $K(S)$ is finite: it fits in the short exact sequence
\[
0\to R_S^*/R_S^{*2}\to K(S)\to\mathrm{Cl}(R_S)[2]\to 0,
\]
where $R_S\subset K$ is the ring of $S$-integers. We can thus compute $S^{2}(E/K)$ provided that we have computed the local images $\im(\delta_{\mathfrak{p}})$ for $\mathfrak{p}\in S$, an $\mathbb{F}_2$-basis for $K(S)$, and for each basis element its image under the natural maps $K^{*}/K^{*2}\to K_{\mathfrak{p}}^*/K_{\mathfrak{p}}^{*2}$ for $\mathfrak{p}\in S$. To compute the local images we use the following Proposition.
\begin{proposition}\label{prop:dimlocalimage}
    Let $F$ be a finite extension of $\mathbb{Q}_p$ for some prime $p$ of $\mathbb{Q}$. Let $E$ be an elliptic curve over $F$. Then
    \[
    \mathrm{dim}_{\mathbb{F}_2}(E(F)/2E(F))=
    \begin{cases}
        \mathrm{dim}_{\mathbb{F}_2}E(F)[2]+[F:\mathbb{Q}_2] & \text{if } p=2,\\
        \mathrm{dim}_{\mathbb{F}_2}E(F)[2] & \text{if } 2<p<\infty,\\
        \mathrm{dim}_{\mathbb{F}_2}E(F)[2]-1 & \text{if } F=\mathbb{R},\\
        0 & \text{if } F=\mathbb{C}
    \end{cases}
    \]
\end{proposition}
\begin{proof}
    See \cite[Lem. 4.4 \& Lem. 4.8]{stoll2001}.
\end{proof}

\section{$X_0(15)$ and 2-Selmer group computations}\label{sec:2selmergroupsection}
In this section we consider the rank of $X_0(15)$ over imaginary quadratic fields. After discussing the parity of the rank in relation to global root numbers, we discuss the torsion subgroup and some local images of an arbitrary quadratic twist $X_0(15)_d$. These local images are then used for various $2$-Selmer group computations, resulting in a proof of Theorem \ref{thm:mainrank}.

\begin{conjecture}[Parity conjecture]
    For an elliptic curve $E$ defined over a number field $K$ we have
    \[
    w(E/K)=(-1)^{\rank(E/K)}.
    \]
\end{conjecture}
Here $w(E/K)\in\{\pm 1\}$ is the global root number of $E/K$, which is the sign of the conjectural functional equation of the $L$-function of $E/K$. The parity conjecture is weaker than the Birch and Swinnerton-Dyer conjecture. We refer to \cite{tdokchitser2013} for an overview on how the parity conjecture relates to various other conjectures, like finiteness of $\Sh$.

Let us fix the notation $E=X_0(15)$ for the remainder of this section. A model for $E$ is given by $v^2+uv+v=u^3+u^2-10u-10$, see \cite[Elliptic curve 15.a5]{lmfdb}. The coordinate change $(x,y)=(4u,8v-4u-4)$ brings the equation in the more desirable form for the purpose of $2$-descent, yielding the equation
\[
E: y^2=(x+13)(x+4)(x-12).
\]
\begin{proposition}\label{prop:globalrootnumber}
    Let $d$ be positive square-free integer. Then
    \[
        w(E/\mathbb{Q}(\sqrt{-d}))=1\quad\Leftrightarrow\quad d\equiv 0,1,2,3,4,5,8,12\bmod 15.
    \]
\end{proposition}
\begin{proof}
    Let $K=\mathbb{Q}(\sqrt{-d})$. As $E/\mathbb{Q}$ is semistable, \cite[Thm. 3]{vdokchitser2005} implies that
    \[
        w(E/K)=(-1)^{1+s},
    \]
    where $s$ is the number of finite primes of $K$ at which $E$ has split multiplicative reduction. Over the rationals, $E$ has non-split multiplicative reduction at $3$, split multiplicative reduction at $5$, and good reduction elsewhere. We see that a prime of $K$ over $3$ has split multiplicative reduction precisely when the residue field degree is $2$, i.e. when $3$ is inert in $K$, in which case there is one such prime. As any prime of $K$ over $5$ has split multiplicative reduction, we have
    \begin{align*}
        \#&\text{split multiplicative primes of }K\text{ over } 3 =
        \begin{cases}
            1, & \text{ if } \left(\tfrac{-d}{3}\right)=-1 \\
            0, & \text{otherwise}
        \end{cases}\\
        \#&\text{split multiplicative primes of }K\text{ over } 5 =
        \begin{cases}
            2, & \text{ if } \left(\tfrac{-d}{5}\right)=1 \\
            1, & \text{otherwise}
        \end{cases}
    \end{align*}
The result follows by going over the cases for $d\bmod 15$.
\end{proof}
The parity conjecture together with Proposition \ref{prop:globalrootnumber} thus immediately yields the following.
\begin{corollary}\label{conj:conjecturalparity}
    Let $d$ be a positive square-free integer and assume the parity conjecture. Then
    \[
    \rank(E/\mathbb{Q}(\sqrt{-d}))\equiv 0\bmod 2\quad\Leftrightarrow\quad d\equiv 0,1,2,3,4,5,8,12\bmod 15.
    \]
\end{corollary}
In particular, assuming the parity conjecture, we see that $E$ has non-zero rank over $\mathbb{Q}(\sqrt{-d})$ when $d\equiv 6,7,9,10,11,13,14\bmod 15$, so \cite[Thm. 1.1]{caraianinewton2023modularity} cannot be used to prove modularity of elliptic curves over $\mathbb{Q}(\sqrt{-d})$ for those $d$. For this reason we restrict to the congruence conditions of Corollary \ref{conj:conjecturalparity} whenever we attempt to prove rank $0$ for a specific $d$. \\

\noindent Let $d\in\mathbb{Q}^{*}$ be a non-square. Then we have the well-known relation
\[
\rank(E/\mathbb{Q})+\rank(E_d/\mathbb{Q})=\rank(E/\mathbb{Q}(\sqrt{d})),
\]
corresponding to a decomposition of $E(\mathbb{Q}(\sqrt{d}))\otimes \mathbb{Q}$ into eigenspaces for the action of $\gal(\mathbb{Q}(\sqrt{d})/\mathbb{Q})$, which we will use freely throughout the rest of the paper. In section \ref{sec:shasection} we will also encounter a proof, see Proposition \ref{prop:rankidentitytwist}. In Proposition \ref{prop:basicrankandshas} we will see that $\rank(E/\mathbb{Q})=0$, so $\rank(E_d/\mathbb{Q})=\rank(E/\mathbb{Q}(\sqrt{d}))$. It follows that we can obtain upper bounds on $\rank(E/\mathbb{Q}(\sqrt{d}))$ by doing a $2$-Selmer group computation of $E_{d}/\mathbb{Q}$.
We start by computing the torsion subgroup and the local images over $2,3,5$ and $\infty$ for any such twist. Then we focus on the main object of interest, the twist $E_{-p}$ for a prime $p$.\\

For a number field $K$ and $d\in K^{*}$, set $e_1=-13$, $e_2=-4$ and $e_3=12$ as in the explicit $2$-descent case of section \ref{sec:2descentsection} corresponding to rational $2$-torsion. Then the Kummer-embedding for $E_d$ as in \eqref{eq:explicitkummer} is given as follows.
\begin{align}\label{eq:kummerx015}
\begin{split}
E_{d}(K)/2E_{d}(K)&\xrightarrow{\delta}\ker\left(\bigoplus_{i=1}^3K^{*}/K^{*2}\to K^{*}/K^{*2}\right),\\
(a,b)&\mapsto d(a+13,a+4,a-12),\quad\text{ if }b\neq 0\\
(-13,0)&\mapsto (1,-d,-d),\\
(-4,0)&\mapsto (d,-1,-d),\\
(12,0)&\mapsto (d,d,1).
\end{split}
\end{align}

\begin{proposition}\label{prop:torsiontwists}
    Let $d\in\mathbb{Z}$ be square-free. Then
    \[
    E_d(\mathbb{Q})_{\mathrm{tors}}\cong
    \begin{dcases*}
        \mathbb{Z}/4\mathbb{Z}\times\mathbb{Z}/2\mathbb{Z} & if $d\in\{1,-1\}$ \\
        \mathbb{Z}/2\mathbb{Z}\times\mathbb{Z}/2\mathbb{Z} & otherwise
    \end{dcases*}
    \]
The points of order $4$ are $(-8,\pm 20)$ and $(32,\pm 180)$ when $d=1$, and $(2,\pm 30)$ and $(-28,120)$ when $d=-1$.
\end{proposition}
\begin{proof}
    The twist $E_{d}$ has good reduction over $\mathbb{Q}_7(\sqrt{7})$, so we have injections
    \[
    E_{d}(\mathbb{Q})[m]\xhookrightarrow{}E_{d}(\mathbb{Q}(\sqrt{7}))[m]\xhookrightarrow{}\tilde{E_{d}}(\mathbb{F}_7)
    \]
    for any positive integer $m$ coprime to $7$. Here the reduction $\tilde{E_{d}}$ is with respect to $\mathbb{Q}_7(\sqrt{7})$. There are two options for this reduction, depending on the class of $d$ in $F^{*}/F^{*2}$ for $F=\mathbb{Q}_7(\sqrt{7})$. In both cases we have $\#\tilde{E_{d}}(\mathbb{F}_7)=8$, which implies that $\#E_{d}(\mathbb{Q})[2^{\infty}]\leq 8$ and $E_{d}(\mathbb{Q})[p]=0$ for all primes $p\neq 2,7$. Similarly, $E_{d}$ has good reduction over $\mathbb{Q}_{11}(\sqrt{11})$ and both possibilities for the reduction have exactly $16$ rational points, so we can also conclude that $E_{d}(\mathbb{Q})[7]=0$. It follows that $\#E_{d}(\mathbb{Q})_{\mathrm{tors}}\leq 8$.\\
    As $E_{d}(\mathbb{Q})[2]\cong\mathbb{Z}/2\mathbb{Z}\times\mathbb{Z}/2\mathbb{Z}$ the torsion subgroup $E(\mathbb{Q})_{\mathrm{tors}}$ is isomorphic to either $\mathbb{Z}/4\mathbb{Z}\times\mathbb{Z}/2\mathbb{Z}$ or $\mathbb{Z}/2\mathbb{Z}\times\mathbb{Z}/2\mathbb{Z}$. The latter case occurs precisely when $E_{d}(\mathbb{Q})[2]\to E_{d}(\mathbb{Q})/2E_{d}(\mathbb{Q})$ is injective, which because of \eqref{eq:kummerx015} is equivalent to $(1,-d,-d)$ and $(d,d,1)$ being linearly independent in $\bigoplus_{i=1}^3\mathbb{Q}^*/\mathbb{Q}^{*2}$, which is the case precisely when $d\neq \pm 1$.\\
    Finally, when $d=1$ we see that $(-8,\pm 20)$ and $(32,\pm 180)$ are indeed points on the curve, which have order $4$ as their doubles equal $(12,0)$. Similarly the doubles of the rational points $(2,\pm 30)$ and $(-28,120)$ on $E_{-1}$ equal $(-13,0)$.
\end{proof}
If one wishes to crack a nut with a sledgehammer, one can also appeal to Mazur's theorem on torsion subgroups over $\mathbb{Q}$ as an alternative to the reduction arguments in the proof of Proposition \ref{prop:torsiontwists}. Then \eqref{eq:kummerx015} is also used to exclude the $\mathbb{Z}/8\mathbb{Z}\times\mathbb{Z}/2\mathbb{Z}$ case.\\

\noindent Let us now prepare for various $2$-Selmer group computations by providing all possible images of the map
\[
E_{d}(\mathbb{Q}_p)/2E_{d}(\mathbb{Q}_p)\xrightarrow{\delta_p}\ker\left(\bigoplus_{i=1}^3\mathbb{Q}_p^{*}/\mathbb{Q}_p^{*2}\to \mathbb{Q}_p^{*}/\mathbb{Q}_p^{*2}\right)
\]
for $p=2,3,5,\infty$. From \eqref{eq:kummerx015} we see that the local image $\im(\delta_p)$ depends only on the class of $d$ in $\mathbb{Q}_p^{*}/\mathbb{Q}_p^{*2}$. We have $\mathbb{Q}_2^*/\mathbb{Q}_2^{*2}=\langle 2,-1,3\rangle$, $\mathbb{Q}_3^*/\mathbb{Q}_3^{*2}=\langle 3,-1\rangle$, $\mathbb{Q}_5^*/\mathbb{Q}_5^{*2}=\langle 5,2\rangle$ and $\mathbb{R}^*/\mathbb{R}^{*2}=\langle -1\rangle$. The local images over $2$ then have the following $\mathbb{F}_2$-bases.
\begin{equation}\label{eq:localimages2}
\begin{gathered}
    \begin{array}{|c|c|c|c|c|}
    \hline
    d\in\mathbb{Q}_2^*/\mathbb{Q}_2^{*2} & 1 & -1 & 3 & -3 \\
    \hline
    \multirow{3}{*}{$\im(\delta_2)$}  & (1,-1,-1) & (-1,-1,1) & (1,-3,-3) & (1,3,3)\\
     & (1,3,3) & (1,-3,-3) & (3,3,1) & (-3,-3,1)\\
     & (-3,1,-3) & (1,2,2) & (1,6,6) & (1,-1,-1)\\
     \hhline{|=|=|=|=|=|}
    d\in\mathbb{Q}_2^*/\mathbb{Q}_2^{*2} & 2 & -2 & 6 & -6 \\
    \hline
    \multirow{3}{*}{$\im(\delta_2)$} & (1,-2,-2) & (1,2,2) & (1,-3,-3) & (1,-1,-1) \\
     & (2,2,1) & (-2,-2,1) & (1,-2,-2) & (1,6,6) \\
     & (1,-1,-1) & (1,3,3) & (6,6,1) & (-6,-6,1) \\
    \hline
    \end{array}
\end{gathered}
\end{equation}
The other images are as follows.
\begin{equation}\label{eq:localimagesrest}
\begin{gathered}
    \begin{array}{|c|c|c|c|c|}
    \hline
    d\in\mathbb{Q}_3^*/\mathbb{Q}_3^{*2} & 1 & -1 & 3 & -3 \\
    \hline
    \multirow{2}{*}{$\im(\delta_3)$} &(1,-1,-1) & (-1,-1,1) & (1,-3,-3) & (1,3,3) \\
     & (-1,-1,1) & (-3,-3,1) & (3,3,1) & (-3,-3,1) \\
    \hline
    \end{array}
\\
    \begin{array}{|c|c|c|c|c|}
    \hline
    d\in\mathbb{Q}_5^*/\mathbb{Q}_5^{*2} & 1 & 2 & 5 & 10 \\
    \hline
    \multirow{2}{*}{$\im(\delta_5)$} & (5,1,5) & (1,2,2) & (1,5,5) & (1,10,10)\\
     & (2,1,2) & (2,2,1) & (5,5,1) & (10,10,1) \\
    \hline
    \end{array}
\\
    \begin{array}{|c|c|c|c|c|}
    \hline
    d\in\mathbb{R}^*/\mathbb{R}^{*2} & 1 & -1 \\
    \hline
    \text{}\im(\delta_{\infty}) & (1,-1,-1) & (-1,-1,1)\\
    \hline
    \end{array}
\end{gathered}
\end{equation}
    For the justification of these local images, observe that Proposition \ref{prop:dimlocalimage} gives us the $\mathbb{F}_2$-dimensions, and one sees quickly that the given elements are linearly independent in all cases, so it remains to show that they are actually in the respective local images. The $2$-dimensional image of $E_{d}(\mathbb{Q})_{\mathrm{tors}}$ yields points in the local images as well, which depending on $p$ and $d$ have to be supplemented by additional local points. It turns out these can be found by taking for $x$ an integer $-6\leq x\leq 9$. For example when $p=2$ and $d=1$, we have the elements $(1,-1,-1)$ and $(-3,1,-3)$ from the global torsion (the global $4$-torsion point $(32,180)$ has image $(5,1,5)$), and for $x=8$ we have in $\mathbb{Q}_2^*/\mathbb{Q}_2^{*2}$
    \[
    (x+13)(x+4)(x-12)=21\cdot 12\cdot -4=1,
    \]
    so we obtain a point $(8,?)\in E(\mathbb{Q}_2)$, which has image $(-3,3,-1)$. The other cases are left to the reader.
\begin{proposition}\label{prop:basicrankandshas}
    For $d\in\{1,-1,-2,-3,-5\}$ we have
    \[
    \rank(E_{d}/\mathbb{Q})=0,\quad\text{ and }\quad\Sh(E_{d}/\mathbb{Q})[2]=0.
    \]
\end{proposition}
\begin{proof}
    From \eqref{eq:selmerranksha} we have
    \[
    \rank(E_{d}/\mathbb{Q})+\dim_{\mathbb{F}_2}\Sh(E_{d}/\mathbb{Q})[2]= \mathrm{dim}_{\mathbb{F}_2}S^2(E_{d}/\mathbb{Q})-2,
    \]
    so the result follows by proving that $\mathrm{dim}_{\mathbb{F}_2}S^2(E_{d}/\mathbb{Q})=2$. The twist $E_{d}$ has good reduction at primes $>5$, and bad (in fact multiplicative) reduction at $3$ and $5$, so we take $S=\{2,3,5,\infty\}$. Then $\mathbb{Q}(S)=\langle -1,2,3,5\rangle\subset\mathbb{Q}^*/\mathbb{Q}^{*2}$. Recall that $\delta$ and the local versions $\delta_p$ for $p\in S$ fit in the commutative diagram
    \[
\begin{tikzcd}
E_{d}(\mathbb{Q})/2E_{d}(\mathbb{Q}) \arrow[r, "\delta"] \arrow[d]                                    & \ker\left(\displaystyle\bigoplus_{i=1}^3\mathbb{Q}^*/\mathbb{Q}^{*2}\to \mathbb{Q}^*/\mathbb{Q}^{*2}\right) \arrow[d, "\iota_p"]                                                    \\
E_{d}(\mathbb{Q}_p)/2E_{d}(\mathbb{Q}_p) \arrow[r, "\delta_p"] & \ker\left(\displaystyle\bigoplus_{i=1}^3\mathbb{Q}_p^*/\mathbb{Q}_p^{*2}\to \mathbb{Q}_p^*/\mathbb{Q}_p^{*2}\right),
\end{tikzcd}
    \]
    and that
    \[
        S^2(E_{d}/\mathbb{Q})=\left\{x\in\ker\left(\bigoplus_{i=1}^3\mathbb{Q}(S)\to\mathbb{Q}^*/\mathbb{Q}^{*2}\right):\iota_p(x)\in\im(\delta_p),\,\forall p\in S\right\}.
    \]
    Consider the case $d=1$ and let $x=(x_1,x_2,x_3)\in S^2(E/\mathbb{Q})$. Then from \eqref{eq:localimages2} we see that $\iota_2(x)\in\im(\delta_2)$ implies that $x_i\in\langle -1,3,5\rangle$ for all $i$. Similarly, considering the $3$-adic image $\im(\delta_3)$ implies $x_i\in\langle -1,2,5\rangle$ for all $i$, hence for each $i$ we have $x_i\in\langle -1,3,5\rangle\cap\langle -1,2,5\rangle=\langle -1,5\rangle$. As $5=-3$ in $\mathbb{Q}_2^{*}/\mathbb{Q}_2^{*2}$, the $2$-adic image implies $x_1\in\langle 5\rangle$, and from the $5$-adic image we deduce $x_2\in\langle -1\rangle$. This leaves $4$ options in total for $x$, proving that indeed $\mathrm{dim}_{\mathbb{F}_2}S^2(E/\mathbb{Q})=2$.\\
    Now suppose that $d=-1$. Then in $S^2(E_{-1}/\mathbb{Q})$ we have the $2$-dimensional subspace generated by $(-1,-1,1)$ and $(15,6,10)$, coming from $E_{-1}(\mathbb{Q}_{\mathrm{tors}})$. This subspace injects onto the subspace of the $2$-adic image spanned by $(-1,-1,1)$ and $(1,-6,-6)$, hence we have $S^2(E_{-1}/\mathbb{Q})=H\oplus \langle (-1,-1,1),(15,6,10)\rangle$ for
    \[
    H=\{x\in S^2(E_{-1}/\mathbb{Q}):\iota_2(x)\in\langle (1,2,2)\rangle\}.
    \]
    Let $x=(x_1,x_2,x_3)\in H$. Then $x_3\in\langle 10\rangle$ by the $3$-adic and real image, which combined with the $2$-adic image forces $x_3=1$. Then $x$ has trivial $2$-adic image, which combined with the $5$-adic image implies $x_2=1$, hence $H=0$ and $S^2(E_{-1}/\mathbb{Q})$ is $2$-dimensional as desired.\\
    The remaining cases are similar and are left to the reader.
\end{proof}
We now focus on when the rank of $E$ over $\mathbb{Q}(\sqrt{-p})$ can be $0$, where $p$ is some prime. In accordance with Corollary \ref{conj:conjecturalparity} we restrict attention to the case $p\equiv 1,2,3,4,5,8\bmod 15$ (no prime can be $\equiv 0,12\bmod 15$), for otherwise we expect an odd, hence non-zero rank. From the $d=1$ case of Proposition \ref{prop:basicrankandshas} we have
\begin{equation}\label{eq:rankrel-p}
\rank(E/\mathbb{Q}(\sqrt{-p}))=\rank(E_{-p}/\mathbb{Q}).
\end{equation}
Proposition \ref{prop:basicrankandshas} thus implies that $\rank(E/\mathbb{Q}(\sqrt{-p}))=0$ for $p\leq 5$, so assume $p>5$. This leaves the possible congruences $p\equiv 1,2,4,8\bmod 15$.

\begin{proposition}\label{prop:2selrat}
    Let $p\equiv 1,2,4,8\bmod 15$ be a prime $>5$. Then the image of $E_{-p}(\mathbb{Q})[2]$ gives two linearly independent elements $(-p,-p,1)$ and $(1,p,p)$ in $S^{2}(E_{-p}/\mathbb{Q})$. The dimension and additional generators for $S^{2}(E_{-p}/\mathbb{Q})$ depend on $p\bmod 8$ and $p\bmod 15$ as follows.
    \[
    \begin{array}{c|c|c|c}
      p \bmod 8 & p\bmod 15  & \mathrm{dim}_{\mathbb{F}_2}S^{2}(E_{-p}/\mathbb{Q}) & \text{additional generators}\\
      \hline
      1 & 1,4 & 4 & (15,6,10),(-1,-1,1)\\
      1 & 2,8 & 4 & (1,2,2),(-1,-1,1)\\
      \text{not } 1 & \text{any} & 2 & \text{none}
    \end{array}
    \]
\end{proposition}
\begin{proof}
    We take $S=\{2,3,5,p,\infty\}$. Then $\mathbb{Q}(S)=\langle -1,2,3,5,p\rangle$. Besides the local images over $2,3,5$ and $\infty$ from \eqref{eq:localimages2} and \eqref{eq:localimagesrest} we also need the image at $p$. We have $\mathbb{Q}_p^{*}/\mathbb{Q}_p^{*2}=\langle r,p\rangle$ with $r\in\mathbb{Z}$ a non-square mod $p$, which we will take to be $-1$ when $p\equiv 3\bmod 4$. The image comes from the $2$-torsion and is given as follows depending on $p\bmod 4$.
    \[
    \im(\delta_p):
    \begin{array}{|c|c|}
    \hline
    p\equiv 1\bmod 4 & p\equiv 3\bmod 4\\
    \hline
    (p,p,1) & (-p,-p,1) \\
    (1,p,p) & (1,p,p) \\
    \hline
    \end{array}
    \]
    Let $H=\{x\in S^{2}(E/K):\iota_p(x)=1\}$. As $\iota_p$ maps the $\delta$-image of $E_{-p}(\mathbb{Q})[2]$ onto the $p$-adic image, we have
    \[
    S^{2}(E/K)=H \oplus \delta(E_{-p}(\mathbb{Q})[2]).
    \]
    \textbf{Case 1:} $p\equiv 1,4\bmod 15$.\\
    In this case, the images of the generators for $\mathbb{Q}(S)$ under the maps $\mathbb{Q}^{*}/\mathbb{Q}^{*2}\to\mathbb{Q}_q^{*}/\mathbb{Q}_q^{*2}$ for $q\in S$, which we will sometimes write as $\im_q(x)$ for $x\in\mathbb{Q}(S)$, are given as follows
    \[
    \begin{array}{c|ccccc}
    & 2 & 3 & 5 & p & \infty \\
    \hline
    -1 & -1 & -1 & 1 &  & -1 \\
    2 & 2 & -1 & 2 &  & 1 \\
    3 & 3 & 3 & 2 &  & 1 \\
    5 & -3 & -1 & 5 & 1 & 1 \\
    p &  & 1 & 1 & p & 1
    \end{array}
    \]
    Here the empty entries depend on $p\equiv 1\bmod 8$. Let $x=(e_1,e_2,e_3)\in H$. Then $e_i\in\langle -1,2,3,5\rangle$ for all $i$ as $\iota_p(x)=1$. Since $\im_3(e_3)$ and $\im_{\infty}(e_3)$ are both trivial, we get $e_3\in\langle 10\rangle$. As $\im_5(e_2)=1$ we have $e_2\in\langle -1,6\rangle$.\\
    If $p\not\equiv 1\bmod 8$, we see that $\im_2(e_3)=-6$ cannot occur, hence $e_3=1$, which when $p\equiv 3,5\bmod 8$ implies $e_2=1$ by considering the $2$-adic image. When $p\equiv 7\mod 8$ considering the $2$-adic image yields $e_2\in\langle -1\rangle$, but $e_2=-1$ cannot happen as it doesn't have trivial $p$-adic image. We conclude that $H=0$ when $p\not\equiv 1\bmod 8$.\\
    If $p\equiv 1\bmod 8$, then one checks that $(15,6,10)\in H$, so we obtain a complement for $\langle (15,6,10)\rangle$ in $H$ by setting $e_3=1$. Then the $2$-adic image of $x$ is contained in $\langle (-1,-1,1)\rangle$, so $e_2\in\langle -1\rangle$. This leaves two options, and as indeed $(-1,-1,1)\in H$, it follows that $H$ is $2$-dimensional and generated by $(15,6,10)$ and $(-1,-1,1)$.\\

    \textbf{Case 2:} $p\equiv 2,8\bmod 15$.\\
    In this case we take $\mathbb{Q}_p^*/\mathbb{Q}_p^{*2}=\langle 3,p\rangle$, and the table of local images of the generators for $\mathbb{Q}(S)$ is then as follows.
    \[
    \begin{array}{c|ccccc}
    & 2 & 3 & 5 & p & \infty \\
    \hline
    -1 & -1 & -1 & 1 &  & -1 \\
    2 & 2 & -1 & 2 &  & 1 \\
    3 & 3 & 3 & 2 &  & 1 \\
    5 & -3 & -1 & 5 & 3 & 1 \\
    p & & -1 & 2 & p & 1
    \end{array}
    \]
    Let $x=(e_1,e_2,e_3)\in H$. Then the $q$-adic images for $q=3,5,p$ imply $e_i\in\langle -1,2\rangle$ for all $i$. By the real image also $e_3\in\langle 2\rangle$. When $p\equiv 3,5\bmod 8$, the $2$-adic image implies $e_3=1$, and then also $e_2=1$ so $H=0$. When $p\equiv 7\bmod 8$ we get $e_3=1$ and $e_2\in\langle -1\rangle$ from the $2$-adic image. But $-1$ has non-trivial $p$-adic image, hence $e_2=1$. It follows that $H=0$ when $p\not\equiv 1\bmod 8$.\\
    When $p\equiv 1\bmod 8$ we see that $(1,2,2)\in H$ so we have a complement for $\langle (1,2,2)\rangle$ in $H$ by setting $e_3=1$. Then $e_2\in\langle -1\rangle$ by the $2$-adic image, and since $(-1,-1,1)\in H$, it follows that $H$ is $2$-dimensional and generated by $(1,2,2)$ and $(-1,-1,1)$.
\end{proof}
It follows that the rank of $E$ over $\mathbb{Q}(\sqrt{-p})$ is at most $2$ when $p\equiv 1\bmod 8$, and equal to $0$ when $p\not\equiv 1\bmod 8$. In particular, combined with Proposition \ref{prop:basicrankandshas} this settles the case $(i)$ of Theorem \ref{thm:mainrank}.\\
We will now focus on the goal of improving the rank bound when $p\equiv 1\bmod 8$. To this end, observe that for a square-free integer $d\neq -p$, the twists $E_{-p}$ and $E_{d}$ are quadratic twists of each other by $-dp$, hence
\begin{equation}\label{eq:rankformulatwist}
\rank(E_{-p}/\mathbb{Q})+\rank(E_{d}/\mathbb{Q})=\rank\left(E_{d}/\mathbb{Q}(\sqrt{-dp})\right).
\end{equation}
By choosing $d$ appropriately, a $2$-Selmer group computation of $E_{d}$ over $\mathbb{Q}(\sqrt{-dp})$ can improve the rank bound from Proposition \ref{prop:2selrat} for some primes $p\equiv 1\bmod 8$. In fact, choosing $d=-1$ will settle case $(ii)$ of Theorem \ref{thm:mainrank}, while $d=-17$ settles case $(iii)$. In section \ref{sec:shasection} these choices for $d$ are clarified.

The remaining $2$-Selmer group computations will make use of R\'edei symbols as in \cite{stevenhagen2021redei}. Let us for convenience extend the definition of minimal ramification as follows. In the notation of \cite[Def. 7.6]{stevenhagen2021redei} we say simply that $L$ is minimally ramified if $K\subset L$ is minimally ramified over $E$, which causes no confusion as $K$ and $E$ are determined by $L$. Additionally, we say that both $\mathbb{Q}(\sqrt{a},\sqrt{\beta})$ and $\mathbb{Q}(\sqrt{b},\sqrt{\beta})$ are minimally ramified if their normal closure over $\mathbb{Q}$, which is $L$, is minimally ramified. 
\begin{proposition}\label{prop:2seloverp}
    Let $p\equiv 1,49\bmod 120$ be a prime. Then
    \[
    \mathrm{dim}_{\mathbb{F}_2}S^{2}(E_{-1}/\mathbb{Q}(\sqrt{p}))=
    \begin{cases*}
        4 & if\, $[-1,10,p]=1$ \\
        2 & if\, $[-1,10,p]=-1$
    \end{cases*}
    \]
\end{proposition}
\begin{proof}
Let $K=\mathbb{Q}(\sqrt{p})$. Let $S$ consist of the primes of $K$ lying over $30$ together with the real primes. Then $|S|=8$, and the completions at the primes in $S$ are given by $\mathbb{Q}_2,\mathbb{Q}_3,\mathbb{Q}_5$ and $\mathbb{R}$, hence we can take the local images from the $d=-1$ cases in \eqref{eq:localimages2} and \eqref{eq:localimagesrest}. They are given as follows.
    \[
    \begin{array}{|c|}
        \hline
        \im(\delta_2) \\
        \hline
        (-1,-1,1)  \\
        (1,-6,-6) \\
        (1,2,2)  \\
        \hline
    \end{array}
    \,\,
    \begin{array}{|c|}
        \hline
        \im(\delta_3) \\
        \hline
        (-1,-1,1) \\
        (3,3,1) \\
        \hline
    \end{array}
    \,\,
    \begin{array}{|c|}
        \hline
        \im(\delta_5) \\
        \hline
        (2,1,2) \\
        (5,1,5) \\
        \hline
    \end{array}
    \,\,
    \begin{array}{|c|}
        \hline
        \im(\delta_{\infty}) \\
        \hline
        (-1,-1,1) \\
        \hline
    \end{array}
    \]\\
    We have the short exact sequence
    \[
    0\to R_S^*/R_S^{*2}\to K(S)\to \mathrm{Cl}(R_S)[2]\to 0,
    \]
    where $R_S=\{x\in K:v_{\mathfrak{p}}(x)\geq 0\text{ for all finite }\mathfrak{p}\in S\}$ is the ring of $S$-integers. Since $\mathrm{Cl}(R_S)$ is a quotient of $\mathrm{Cl}_K$ and $\mathrm{Cl}_K[2]=0$, we have $\mathrm{Cl}(R_S)[2]=0$. The finitely generated abelian group $R_S^{*}$ has rank $7$ and cyclic torsion subgroup of even order, so it follows that $R_S^*/R_S^{*2}$, and hence also $K(S)$, has $\mathbb{F}_2$-dimension $8$.\\
    To find a convenient basis for $K(S)$, observe that for $d\in\{-1,2,3,5\}$ we have $(p,d)_q=1$ for all primes $q$, hence there exists $x_d\in K$ with $N(x_d)\in d\mathbb{Q}^{*2}$ such that $K(\sqrt{x_d})/K$ is minimally ramified. Also denote by $y_d$ the conjugate of $x_d$. Let
    \[
    H:=\langle x_{-1},y_{-1},x_2,y_2,x_3,y_3,x_5,y_5\rangle\subset K^{*}/K^{*2}.
    \]
    The minimal ramification of the extensions $K(\sqrt{x_d})/K$ implies $H\subset K(S)$. In fact equality holds because the norm-induced map $H\to\mathbb{Q}^{*}/\mathbb{Q}^{*2}$ has kernel and image of dimension $4$, so $H$ and $K(S)$ have the same $\mathbb{F}_2$-dimension.
    
    For each $d\in\{-1,2,3,5\}$, we may replace both $x_d$ and $y_d$ with $tx_d$ and $ty_d$, respectively, for some $t$ in the twisting subgroup $T_{p,d}\subset\mathbb{Q}^{*}/\mathbb{Q}^{*2}$ as in \cite[Eq. (43) and Lemma 7.7]{stevenhagen2021redei}. This doesn't change the fact that they form a basis for $K(S)$. These subgroups are given as follows
    \[
    \begin{array}{c|c}
        d & T_{p,d} \\
        \hline
        -1 & \langle -1,p\rangle \\
        2 & \langle 2,p\rangle \\
        3 & \langle -1,-3,p \rangle \\
        5 & \langle 5,p \rangle
    \end{array}
    \]
    We now make various choices, using up part of the elbow room provided by the twisting subgroups. \\
    Let $\mathfrak{p}_{\infty}:K\to\mathbb{R}$ be the infinite prime such that $\mathfrak{p}_{\infty}(x_{-1})>0$ and let $\mathfrak{q}_{\infty}$ be the other infinite prime. Let $\mathfrak{p}_2$ be the prime of $K$ over $2$ that is unramified in $K(\sqrt{x_{-1}})$ and let $\mathfrak{q}_2$ be the conjugate of $\mathfrak{p}_2$. Twist $x_2$ by $2$ if necessary such that $\mathfrak{p}_2$ is unramified in $K(\sqrt{x_2})$. For $r\in\{3,5\}$, let $\mathfrak{p}_r$ be the prime of $K$ over $r$ that splits in $K(\sqrt{x_2})$ and let $\mathfrak{q}_r$ be the other prime over $r$. Twist $x_r$ by $r^{*}$ if necessary such that $\mathfrak{p}_r$ is unramified in $K(\sqrt{x_r})$. Finally, twist $x_3$ by $-1$ if necessary such that $\mathfrak{p}_2$ is unramified in $K(\sqrt{x_3})$. For $d\in\{3,5\}$, we get from $\prod_{\mathfrak{p}}(x_2,y_d)_{\mathfrak{p}}=1$ that $\im_{\mathfrak{p}_2}(x_d)=1$. We now have the following table, where the entries in a column corresponding to a prime $\mathfrak{p}\in S$ are the images of the generators of $K(S)$ under the map $K(S)\to K_{\mathfrak{p}}^{*}/K_{\mathfrak{p}}^{*2}$. We will often write $\im_{\mathfrak{p}}(x)$ for the image of $x\in K(S)$ under this map.
    \begin{equation}\label{2sel17table}
    \begin{array}{c|cccccccc}
    & \mathfrak{p}_2 & \mathfrak{q}_2 & \mathfrak{p}_3 & \mathfrak{q}_3 & \mathfrak{p}_5 & \mathfrak{q}_5 & \mathfrak{p}_{\infty} & \mathfrak{q}_{\infty} \\
    \hline
    x_{-1} &\cellcolor[gray]{0.8} & &\cellcolor[gray]{0.8} & & \cellcolor[gray]{0.8} & & 1 & -1 \\
    y_{-1} & & & & & & & -1 & 1 \\
    x_2 & \cellcolor[gray]{0.8}& & 1 & -1 & 1 & 2 &\cellcolor[gray]{0.8} & \\
    y_2 & & & -1 & 1 & 2 & 1 & & \\
    x_3 & 1 & 3 &\cellcolor[gray]{0.8} & &\cellcolor[gray]{0.8} & &\cellcolor[gray]{0.8} & \\
    y_3 & 3 & 1 & & & & & & \\
    x_5 & 1 & -3 &\cellcolor[gray]{0.8} & &\cellcolor[gray]{0.8} & &\cellcolor[gray]{0.8} & \\
    y_5 & -3 & 1 & & & & & & \\
    \end{array}
    \end{equation}
    For conjugate primes $\mathfrak{p}$ and $\mathfrak{q}$ of $K$ we have $\im_{\mathfrak{p}}(x_d)=\im_{\mathfrak{q}}(y_d)$, and also $\im_{\mathfrak{p}}(x_d)=d\cdot \im_{\mathfrak{p}}(y_d)$. It follows that any $2\times 2$ block of the table corresponding to two conjugate primes and $x_d$ and $y_d$ for some $d\in\{-1,2,3,5\}$ is determined by any one of the four entries. Each gray-colored entry has only two options because $\mathfrak{p}_i$ is unramified in $K(\sqrt{x_d})$ for $i\in\{2,3,5,\infty\}$ and $d\in\{-1,2,3,5\}$, hence there are at most $2^{11}$ possibilities for the table. This upper bound can be reduced to $2^7$ due to various reciprocity relations\footnote{Tests with Magma indicate that all of these $2^7$ options actually occur, meaning that we have found all possible relations. Proving this is not necessary to complete the computation however.}:
    \begin{enumerate}[(a)]
        \item $\im_{\mathfrak{p}_2}(x_{-1})=1$ $\Leftrightarrow$ $\im_{\mathfrak{p}_{\infty}}(x_2)=1$ as $\prod_{\mathfrak{p}}(x_{-1},y_2)_{\mathfrak{p}}=1$.
        \item $\im_{\mathfrak{p}_3}(x_{-1})=1$ $\Leftrightarrow$ $\im_{\mathfrak{p}_{\infty}}(x_3)=1$ as $\prod_{\mathfrak{p}}(x_{-1},y_3)_{\mathfrak{p}}=1$.
        \item $\im_{\mathfrak{p}_5}(x_{-1})=1$ $\Leftrightarrow$ $\im_{\mathfrak{p}_{\infty}}(x_5)=1$ as $\prod_{\mathfrak{p}}(x_{-1},y_5)_{\mathfrak{p}}=1$.
        \item $\im_{\mathfrak{p}_5}(x_3)=1$ $\Leftrightarrow$ $\im_{\mathfrak{p}_3}(x_5)=1$ as $\prod_{\mathfrak{p}}(x_3,y_5)_{\mathfrak{p}}=1$.
   \end{enumerate}
   We now take four cases, depending on $\im_{\mathfrak{p}_2}(x_{-1})$ and $\im_{\mathfrak{p}_5}(x_{-1})$.\\
   
   \textbf{Case 1}: $\im_{\mathfrak{p}_2}(x_{-1})=1$ and $\im_{\mathfrak{p}_5}(x_{-1})=1$.\\
   In this case the table of local images take the following form.
   \begin{equation}
    \begin{array}{c|cccccccc}
    & \mathfrak{p}_2 & \mathfrak{q}_2 & \mathfrak{p}_3 & \mathfrak{q}_3 & \mathfrak{p}_5 & \mathfrak{q}_5 & \mathfrak{p}_{\infty} & \mathfrak{q}_{\infty} \\
    \hline
    x_{-1} & 1 & -1 & & & 1 & 1 & 1 & -1 \\
    y_{-1} & -1 & 1 & & & 1 & 1 & -1 & 1 \\
    x_2 & & & 1 & -1 & 1 & 2 & 1 & 1 \\
    y_2 & & & -1 & 1 & 2 & 1 & 1 & 1\\
    x_3 & 1 & 3 & & & & & & \\
    y_3 & 3 & 1 & & & & & & \\
    x_5 & 1 & -3 & & & & & 1 & 1\\
    y_5 & -3 & 1 & & & & & 1 & 1 \\
    \end{array}
    \end{equation}
    Let $e=(e_1,e_2,e_3)\in S^2(E_{-1}/K)$. As $\iota_{\mathfrak{p}_2}(e)$ and $\iota_{\mathfrak{q}_2}(e)$ are inside $\im(\delta_2)$, of which the first coordinate is contained in $\langle -1,3\rangle\subset\mathbb{Q}_2^*/\mathbb{Q}_2^{*2}$, we see $e_1\in\langle x_{-1},y_{-1},x_3,y_3,x_5,y_5\rangle$. Using that in fact the first coordinate of elements in $\im(\delta_2)$ are in $\langle -1\rangle$ further reduces this to $e_1\in\langle x_{-1},y_{-1},x_3x_5,y_3y_5\rangle$, leaving $2^4$ options. All of these options in fact occur: the conjugate elements $(x_{-1},x_{-1},1)$ and $(y_{-1},y_{-1},1)$ are in $S^2(E_{-1}/K)$. If $\mathrm{im}_{\mathfrak{p}_3}(x_5)=1$ we also find $(x_3x_5,x_2x_3,x_2x_5)$ in $S^2(E_{-1}/K)$ with its conjugate, while if $\mathrm{im}_{\mathfrak{p}_3}(x_5)=-1$ we have $(x_3x_5,y_2x_3,y_2x_5)$ and its conjugate in $S^2(E_{-1}/K)$. Regardless of $\mathrm{im}_{\mathfrak{p}_3}(x_5)$, we see that a complement $H$ inside $S^2(E_{-1}/K)$ to the subspace generated by these four elements is obtained by setting $e_1=1$. The $3$-adic and $5$-adic images then imply $e_2=e_3\in\langle x_{-1},y_{-1},x_2,y_2\rangle$, upon which the $5$-adic and real images imply that $e_2=e_3=1$. It follows that $H=0$ and $S^2(E_{-1}/K)$ is $4$-dimensional.\\

    \textbf{Case 2}: $\im_{\mathfrak{p}_2}(x_{-1})=-3$ and $\im_{\mathfrak{p}_5}(x_{-1})=1$.\\
    The table now takes the following form.
    \begin{equation}
    \begin{array}{c|cccccccc}
    & \mathfrak{p}_2 & \mathfrak{q}_2 & \mathfrak{p}_3 & \mathfrak{q}_3 & \mathfrak{p}_5 & \mathfrak{q}_5 & \mathfrak{p}_{\infty} & \mathfrak{q}_{\infty} \\
    \hline
    x_{-1} & -3 & 3 & & & 1 & 1 & 1 & -1 \\
    y_{-1} & 3 & -3 & & & 1 & 1 & -1 & 1 \\
    x_2 & & & 1 & -1 & 1 & 2 & -1 & -1 \\
    y_2 & & & -1 & 1 & 2 & 1 & -1 & -1\\
    x_3 & 1 & 3 & & & & & & \\
    y_3 & 3 & 1 & & & & & & \\
    x_5 & 1 & -3 & & & & & 1 & 1\\
    y_5 & -3 & 1 & & & & & 1 & 1 \\
    \end{array}
    \end{equation}
    Let $e=(e_1,e_2,e_3)\in S^2(E_{-1}/K)$. The $3$-adic image imply $e_3\in\langle x_{-1},y_{-1},x_2,y_2,x_5,y_5\rangle$, upon which the real image implies $e_3\in\langle x_{-1}y_{-1},x_2y_2,x_5,y_5\rangle$. As the $3$'rd coordinate of elements in $\im(\delta_2)$ is contained in $\langle 2,-3\rangle$, we further conclude that $e_3\in\langle x_2y_2,x_5,y_5\rangle$. Considering both the $\mathfrak{p}_3$ and $\mathfrak{q}_3$-adic image in either case of $\im_{\mathfrak{p}_3}(x_5)$ implies $e_3\in\langle x_2y_2x_5y_5\rangle=\langle 10\rangle$. In $S^2(E_{-1}/K)$ we find the element $(15,6,10)$ (which is the image of $(-28,120)\in E_{-1}(K)[4]$), hence $H=\{(e_1,e_2,e_3)\in S^2(E_{-1}/K): e_3=1\}$ is a complement to $\langle (15,6,10)\rangle$ inside $S^2(E_{-1}/K)$. For $e=(e_1,e_2,e_3)\in H$ we find $e_1=e_2\in\langle x_{-1},y_{-1},x_3,y_3\rangle$ by considering the $2$-adic and $5$-adic images. The $2$-adic image the further implies $e_1=e_2\in\langle x_{-1}y_{-1},x_{-1}x_3y_3\rangle$, upon which the $5$-adic image yields $e_1=e_2\in\langle x_{-1}y_{-1}\rangle=\langle -1\rangle$. Since indeed $(-1,-1,1)\in S^2(E_{-1}/K)$, we conclude that $H=\langle (-1,-1,1)\rangle$ and that $S^2(E_{-1}/K)$ has $\mathbb{F}_2$-dimension $2$.\\

    The other two cases are covered similarly and are left to the reader. The results are as follows: in $S^{2}(E_{-1}/K)$ we always have a $2$-dimensional subspace coming from $E_{-1}(K)[4]$, which is generated by $(-1,-1,1)$ and $(15,6,10)$. A basis for $S^{2}(E_{-1}/K)$ is obtained by adding the following additional generators.
    \pagebreak[2]
    \[
    \begin{array}{c|c|c|c|c}
      \im_{\mathfrak{p}_2}(x_{-1}) & \im_{\mathfrak{p}_5}(x_{-1})  & \im_{\mathfrak{p}_3}(x_5) & \mathrm{dim}_{\mathbb{F}_2}S^{2}(E_ {-1}/\mathbb{Q}(\sqrt{p})) & \text{additional generators}\\
      \hline
      1 & 1 & 1 & 4 & (x_{-1},x_{-1},1), (x_3x_5,x_2x_3,x_2x_5)\\
      1 & 1 & -1 & 4 & (x_{-1},x_{-1},1), (x_3x_5,y_2x_3,y_2x_5)\\
      1 & 2 & 1 & 2 & \textit{none}\\
      1 & 2 & -1 & 2 & \textit{none}\\
      -3 & 1 & 1 & 2 & \textit{none}\\
      -3 & 1 & -1 & 2 & \textit{none}\\
      -3 & 2 & 1 & 4 & (3x_{-1},3x_{-1},1), (x_3x_5,x_2x_3,x_2x_5)\\
      -3 & 2 & -1 & 4 & (3x_{-1},3x_{-1},1), (x_3x_5,y_2x_3,y_2x_5)\\
    \end{array}
    \]
    For $d\in\{2,5\}$ we have that $K(\sqrt{x_d})$ is minimally ramified, hence we have $\im_{\mathfrak{p}_d}(x_{-1})=1$ if and only if $[p,-1,d]=1$. Thus we see that $S^{2}(E_{-1}/K)$ has dimension $2$ if $[p,-1,2][p,-1,5]=-1$, and dimension $4$ if $[p,-1,2][p,-1,5]=1$. Using tri-linearity and reciprocity of R\'edei symbols, we have $[p,-1,2][p,-1,5]=[p,-1,10]=[-1,10,p]$, thus completing the proof.
\end{proof}

\begin{proposition}\label{prop:2seloverpq}
    Let $p,q\equiv 17,113\bmod 120$ be distinct primes. Then
    \[
        \mathrm{dim}_{\mathbb{F}_2}S^{2}(E_{-q}/\mathbb{Q}(\sqrt{pq}))=
        \begin{cases*}
            6 & if\, $[-1,2,pq]=1$ \\
            4 & if\, $[-1,2,pq]=-1$
        \end{cases*}
    \]
\end{proposition}
\begin{proof}
    Let $K=\mathbb{Q}(\sqrt{pq})$. The primes of $K$ where $E_{-q}$ has bad reduction are those over $2$, $3$ and $5$. Note that the good reduction at the prime over $q$ is due to the completion being a ramified quadratic extension of $\mathbb{Q}_q$. The completions at the primes of $K$ over $2,3,5$ and $\infty$ are $\mathbb{Q}_2,\mathbb{Q}_3,\mathbb{Q}_5$ and $\mathbb{R}$, respectively. The local images for these completions are as follows.
    \[
    \begin{array}{|c|}
        \hline
        \im(\delta_2) \\
        \hline
        (-1,-1,1)  \\
        (1,-6,-6) \\
        (1,2,2)  \\
        \hline
    \end{array}
    \,\,
    \begin{array}{|c|}
        \hline
        \im(\delta_3) \\
        \hline
        (1,-1,-1) \\
        (-1,-1,1) \\
        \hline
    \end{array}
    \,\,
    \begin{array}{|c|}
        \hline
        \im(\delta_5) \\
        \hline
        (2,2,1) \\
        (1,2,2) \\
        \hline
    \end{array}
    \,\,
    \begin{array}{|c|}
        \hline
        \im(\delta_{\infty}) \\
        \hline
        (-1,-1,1) \\
        \hline
    \end{array}
    \]
    From the $3$-adic and $5$-adic images, we see that in fact
    \[
    S^{2}(E_{-q}/K)\subset\ker\left(\bigoplus_{i=1}^3 K(T)\to K(T)\right),
    \]
    where $T$ consists just of the primes of $K$ over $2$ together with the real primes.\\
    For $d\in\{-1,2\}$ we have $(pq,d)_r=1$ for all primes $r$ so, there exists $x_d\in K$ with $N(x_d)\in d\mathbb{Q}^{*2}$ such that $K(\sqrt{x_d})/K$ is minimally ramified. Let $y_d$ be the conjugate of $x_d$ in $K$. Consider
    \[
    H:=\langle x_{-1},y_{-1},x_2,y_2,q\rangle\subset K^{*}/K^{*2}.
    \]
    Then $H$ has $\mathbb{F}_2$-dimension $5$ as the norm induced map $H\to\mathbb{Q}^{*}/\mathbb{Q}^{*2}$ has $2$-dimensional image and a $3$-dimensional kernel. The minimal ramification of the extensions $K(\sqrt{x_d})$ for $d=-1,2$, and the fact that $q$ is ramified in $K$, implies that $H\subset K(T)$. This inclusion is an equality, as the short exact sequence
    \[
    0\to R_T^*/R_T^{*2}\to K(T)\to \mathrm{Cl}(R_T)[2]\to 0,
    \]
    implies $\dim_{\mathbb{F}_2}K(T)\leq 5$ as $R_T^*/R_T^{*2}$ has dimension $4$ and $\mathrm{Cl}(R_T)[2]$ has at most dimension $1$ as as $\mathrm{Cl}_K[2]\cong\mathbb{Z}/2\mathbb{Z}$ and $\mathrm{Cl}(R_T)$ is a quotient of $\mathrm{Cl}_K$.\\
    The twisting subgroups are now as follows.
        \[
    \begin{array}{c|c}
        d & T_{pq,d} \\
        \hline
        -1 & \langle -1,p,q\rangle \\
        2 & \langle 2,p,q\rangle \\
    \end{array}
    \]
    Let $\mathfrak{p}_{\infty}:K\to\mathbb{R}$ be the infinite prime such that $\mathfrak{p}_{\infty}(x_{-1})>0$ and let $\mathfrak{q}_{\infty}$ be the other infinite prime. Let $\mathfrak{p}_2$ be the prime of $K$ over $2$ that is unramified in $K(\sqrt{x_{-1}})$ and let $\mathfrak{q}_2$ be the other prime over $2$. Twist $x_2$ by $2$ if necessary such that $\mathfrak{p}_2$ is unramified in $K(\sqrt{x_2})$. We then have the following table of local images.
    \[
    \begin{array}{c|cccc}
        & \mathfrak{p}_2 & \mathfrak{q}_2 & \mathfrak{p}_{\infty} & \mathfrak{q}_{\infty} \\
        \hline
        x_{-1} &\cellcolor[gray]{0.8} & & 1 & -1 \\
        y_{-1} & & & -1 & 1 \\
        x_2 &\cellcolor[gray]{0.8} & &\cellcolor[gray]{0.8} & \\
        y_2 & & & & \\
        q & 1 & 1 & 1 & 1 
    \end{array}
    \]
    As $K(\sqrt{x_{-1}})/K$ is minimally ramified, we find that $\im_{\mathfrak{p}_2}(x_{-1})=1$ precisely when $[pq,-1,2]=1$. Similarly, working with $K(\sqrt{x_2})$ we find that $\im_{\mathfrak{p}_{\infty}}(x_2)=1$ precisely when $[pq,2,-1]=1$. By R\'edei reciprocity we have $[pq,2,-1]=[pq,-1,2]$, giving a relation between the two conditions, and by the reciprocity again they both also equal $[-1,2,pq]$.\\

    \textbf{Case 1}: $[-1,2,pq]=1$.\\
    In this case the table takes the form
    \[
    \begin{array}{c|cccc}
        & \mathfrak{p}_2 & \mathfrak{q}_2 & \mathfrak{p}_{\infty} & \mathfrak{q}_{\infty} \\
        \hline
        x_{-1} & 1 & -1 & 1 & -1 \\
        y_{-1} & -1 & 1 & -1 & 1 \\
        x_2 & a & 2a & 1 & 1 \\
        y_2 & 2a & a & 1 & 1 \\
        q & 1 & 1 & 1 & 1 
    \end{array}
    \]
    with $a\in\langle -3\rangle\subset\mathbb{Q}_2^{*}/\mathbb{Q}_2^{*2}$. Let $(e_1,e_2,e_3)\in S^{2}(E_{-q}/K)$. Then the real images imply that $e_3\in\langle x_2,y_2,q\rangle$, while the first coordinate of the $2$-adic images imply $e_1\in\langle x_{-1},y_{-1},q\rangle$. As $e_2$ is determined by $e_1$ and $e_3$, this leaves $2^6$ options, all of which are in fact in $S^{2}(E_{-q}/K)$. Thus $S^{2}(E_{-q}/K)$ is $6$-dimensional.\\

    \textbf{Case 2}: $[-1,2,pq]=-1$.\\
    In this case the table takes the form
    \[
    \begin{array}{c|cccc}
        & \mathfrak{p}_2 & \mathfrak{q}_2 & \mathfrak{p}_{\infty} & \mathfrak{q}_{\infty} \\
        \hline
        x_{-1} & -3 & 3 & 1 & -1 \\
        y_{-1} & 3 & -3 & -1 & 1 \\
        x_2 & a & 2a & -1 & -1 \\
        y_2 & 2a & a & -1 & -1 \\
        q & 1 & 1 & 1 & 1 
    \end{array}
    \]
    with $a\in\langle -3\rangle\subset\mathbb{Q}_2^{*}/\mathbb{Q}_2^{*2}$. Let $(e_1,e_2,e_3)\in S^{2}(E_{-q}/K)$. Then the real images imply that $e_3\in\langle 2,q\rangle$, while the first coordinate of the $2$-adic images imply $e_1\in\langle -1,q\rangle$. This leaves $2^4$ options, all of which are in $S^{2}(E_{-q}/K)$. Thus $S^{2}(E_{-q}/K)$ is $4$-dimensional.
\end{proof}
As final ingredient for the proof of Theorem \ref{thm:mainrank}, let us state the following version of Chebotar\"ev's density theorem. We refer to \cite[Ch. I, \S 2.2]{serre1998}.
\begin{theorem}[Chebotar\"ev's density theorem]\label{thm:chebo}
Let $K/\mathbb{Q}$ be a finite Galois extension and let $C\subset\gal(K/\mathbb{Q})$ be a non-empty subset closed under conjugation. Then
\[
\lim_{X\to\infty}\frac{\#\{p\leq X : p \text{ is unramified in }K\text{ and }\mathrm{Frob}_p\in C\}}{\#\{p\leq X\}}=\frac{\# C}{[K:\mathbb{Q}]}.
\]
\end{theorem}
Here all instances of $p$ assume that $p$ is a prime, and $\mathrm{Frob}_p$ is some choice of $\mathrm{Frob}_{\mathfrak{p}}$ for $\mathfrak{p}$ a prime of $K$ over $p$. Different choices for $\mathrm{Frob}_p$ are conjugate, so as $C$ is closed under conjugation, the condition $\mathrm{Frob}_p\in C$ is independent of the choice of $\mathrm{Frob}_p$.
\begin{proof}[Proof of Theorem \ref{thm:mainrank}]
    Recall from \eqref{eq:rankrel-p} the identity
    \begin{equation}\label{eq:rankrel-p2}
    \rank(E/\mathbb{Q}(\sqrt{-p}))=\rank(E_{-p}/\mathbb{Q}).
    \end{equation}
    We prove that $\rank(E/\mathbb{Q}(\sqrt{-p}))=0$ by showing that $\rank(E_{-p}/\mathbb{Q})=0$. Suppose that $p\equiv 1,2,3,4,5,8\bmod 15$. Then $\rank(E_{-p}/\mathbb{Q})=0$ follows from Proposition \ref{prop:basicrankandshas} if $p\leq 5$, and from Proposition \ref{prop:2selrat} if $p>5$.\\
    If $p\equiv 1\bmod 8$, $p\equiv 1,4\bmod 15$, and $[-1,10,p]=-1$, then applying both Proposition \ref{prop:basicrankandshas} and \eqref{eq:rankformulatwist} for $d=-1$ we obtain
    \[
        \rank(E_{-p}/\mathbb{Q})=\rank(E_{-1}/\mathbb{Q}(\sqrt{p})).
    \]
    The right hand side is $0$ by Proposition \ref{prop:2seloverp}.\\
    Lastly, assume that $p\equiv 1\bmod 8$, $p\equiv 2,8\bmod 15$, and $[-1,2,p]=1$. Then $p\neq 17$ as $[-1,2,17]=-1$, and as $[-1,2,17p]=-1$, we get from Proposition \ref{prop:2seloverpq} and \eqref{eq:rankformulatwist} with $d=-17$ the inequality
    \[
        \rank(E_{-p}/\mathbb{Q})+\rank(E_{-17}/\mathbb{Q})\leq 2.
    \]
    In fact we have $\rank(E_{-17}/\mathbb{Q})=2$: the points $(\tfrac{-2246}{17},\tfrac{-103950}{17^2})$ and $(\tfrac{52}{17},\tfrac{832}{17^2})$ have images $(1,2,2)$ and $(-1,-1,1)$ in $S^{2}(E_{-17}/\mathbb{Q})$, respectively. We conclude that $\rank(E_{-p}/\mathbb{Q})=0$.\\

    It remains to prove the density statement.  Note that it suffices to show that
    \begin{equation}\label{eq:maindens}
        \lim_{X\to\infty}\frac{\#\{p\leq X: p \text{ satisfies } (i), (ii) \text{ or } (iii)\}}{\#\{p\leq X\}}=\frac{7}{16},
    \end{equation}
    because from Theorem \ref{thm:chebo} applied to $\mathbb{Q}(\zeta_{15})$ and $C=\{\overline{1},\overline{2},\overline{4},\overline{8}\}\subset(\mathbb{Z}/15\mathbb{Z})^{*}=\mathrm{Gal}(\mathbb{Q}(\zeta_{15})/\mathbb{Q})$ we deduce
    \[
        \lim_{X\to\infty}\frac{\#\{p\leq X: p \equiv 1,2,3,4,5,8\bmod 15\}}{\#\{p\leq X\}}=\frac{1}{2}.
    \]
    Applying Theorem \ref{thm:chebo} to $\mathbb{Q}(\zeta_{120})$ we deduce
    \begin{equation}\label{eq:densi}
        \lim_{X\to\infty}\frac{\#\{p\leq X:p \text{ satisfies } (i)\}}{\#\{p\leq X\}}=\frac{4\cdot 3}{\varphi(120)}=\frac{3}{8}
    \end{equation}
    Let $E=\mathbb{Q}(\sqrt{-1},\sqrt{10})$ and $F=E(\sqrt{3+\sqrt{-1}})$. Then $F$ is minimally ramified and can be used to compute the R\'edei symbol $[-1,10,p]$ provided it defined, which is the case if and only if $p\equiv 1,9\bmod 40$.\\
    Consider the Galois extenion $F(\zeta_{24})/\mathbb{Q}$ of degree $32$. Let $\sigma\in\gal(F(\zeta_{24})/\mathbb{Q})$ be the generator of the order $2$ subgroup corresponding to the intermediate field $E(\zeta_{24})$, and let $Z_{\mathfrak{p}}$ be the decomposition field of a prime $\mathfrak{p}$ of $F(\zeta_{24})$ over $p$. The prime $p$ splits completely in $E(\zeta_{24})$ precisely when $p\equiv 1,49\bmod 120$, hence
    \[
    p\text{ satisfies } (ii)\,\Leftrightarrow\,Z_{\mathfrak{p}}=E(\zeta_{24})=\,\Leftrightarrow\,\mathrm{Frob}_p=\sigma.
    \]
    Theorem \ref{thm:chebo} thus implies
    \begin{equation}\label{eq:densii}
        \lim_{X\to\infty}\frac{\#\{p\leq X:p \text{ satisfies } (ii)\}}{\#\{p\leq X\}}=\frac{1}{32}
    \end{equation}
    For condition $(iii)$, let $E'=\mathbb{Q}(\sqrt{-1},\sqrt{2})=\mathbb{Q}(\zeta_8)$, and $F'=E'(\sqrt{1+\sqrt{-1}})$. The standard isomorphism
    \[
        \gal(F'(\zeta_{15})/\mathbb{Q})\xrightarrow{\sim}\gal(F'/\mathbb{Q})\times\gal(\mathbb{Q}(\zeta_{15})/\mathbb{Q})=\gal(F'/\mathbb{Q})\times(\mathbb{Z}/15\mathbb{Z})^{*},
    \]
    maps a Frobenius element $\mathrm{Frob}_p$ over $p$ to $(\mathrm{Frob}_p,p\bmod 15)$. Let $C\subset \gal(F'(\zeta_{15})/\mathbb{Q})$ be the subset corresponding under this isomorphism with $1\times\{\overline{2},\overline{8}\}$. It follows that $\mathrm{Frob}_p\in C$ if and only if $p\equiv 17,113\bmod 120$ and $[-1,2,p]=1$, hence with Theorem \ref{thm:chebo} we obtain
    \begin{equation}\label{eq:densiii}
        \lim_{X\to\infty}\frac{\#\{p\leq X:p \text{ satisfies } (iii)\}}{\#\{p\leq X\}}=\frac{2}{64}=\frac{1}{32}.
    \end{equation}
    As the conditions $(i)$, $(ii)$ and $(iii)$ are mutually exclusive, \eqref{eq:maindens} follows by taking the sum of \eqref{eq:densi}, \eqref{eq:densii} and \eqref{eq:densiii}, thus completing the proof.
\end{proof}
\begin{remark}
    Any prime $q\equiv 17,113\bmod 120$ such that $\rank(E_{-q}/\mathbb{Q})=2$ can be used in the proof of \ref{thm:mainrank}(iii). Possible alternatives for the choice of $q=17$ are for example $q=233$, or $q=617$. A priori the prime $q$ must also satisfy $[-1,2,q]=-1$, but after having a proof of Theorem \ref{thm:mainrank}(iii) we know this follows from the rank being non-zero.
\end{remark}

\section{Relation to $4$-descent}\label{sec:shasection}
\noindent In the previous section we encountered various instances where the rank bound of $X_0(15)_{-p}/\mathbb{Q}$ obtained from a $2$-Selmer group computation could be improved using $2$-Selmer group computations over suitable quadratic fields. A natural question is how this improvement is related to a $4$-descent. In this section we prove that the rank bounds are exactly as strong as those obtained from a $4$-descent. In fact we will show a stronger result relating higher descents. After a lemma we recall parts of \cite[\S 4]{Bruin_2004} concerning visualisation techniques of $\Sh[2]$, and explain how these techniques can improve on a $2$-descent. The main result of this section, Theorem \ref{thm:mainsha}, gives a condition under which the rank bounds obtained from visualisation of $\Sh[2]$ are in fact equivalent to those obtained from a $4$-descent. We then use this result to show that the rank bounds obtained in section $3$ are equivalent to those obtained from performing a $4$-descent on $X_0(15)_{-p}$ over $\mathbb{Q}$.\\

\noindent Throughout this section, let $E$ be an elliptic curve defined over a number field $K$. 
\begin{lemma}\label{lem:normcor}
    For $L/K$ a finite extension we have the following commutative diagram.
    \[
    \begin{tikzcd}
        M_L' \arrow[r, "\sim"] \arrow[d, "N"] & {H^1(L,E[2])} \arrow[d, "\mathrm{cor}"] \\
        M_K' \arrow[r, "\sim"]                & {H^1(K,E[2])}                          
    \end{tikzcd}
    \]
\end{lemma}
\noindent Here $N$ is induced from the relative norm map $N_{M_L/M_K}:M_L\to M_K$ and the horizontal isomorphisms are as in section \ref{sec:2descentsection}. Note that $M_L$ is a finite free $M_K$-algebra of rank $[L:K]$ as $M_L=M_K\otimes_K L$, hence we indeed have said norm map.
\begin{proof}
     By transitivity of the norm, see \cite[Prop. III.9.4]{bourbaki1998algebra}, we have a commutative diagram of norm maps.
    \[
    \begin{tikzcd}
               & M_L^* \arrow[ld] \arrow[rd] &              \\
M_K^* \arrow[rd] &                           & L^* \arrow[ld] \\
               & K^*                         &             
\end{tikzcd}
    \]
    Modding out by squares, we obtain an induced map $N:M_L'\to M_K'$. For the commutativity, let $\overline{M}_K=\overline{K}[t]/f(t)$ and similarly for $L$. Letting $G_K$ act trivially on $t$ we give $\overline{M}_K$ a $G_K$-module structure in the obvious way. Then $M_K=\overline{M}_K^{G_K}$, and the norm map $N_{M_L/M_K}:M_L\to M_K$ equals $\mathrm{cor}:H^0(L,\overline{M}_L)\to H^0(K,\overline{M}_K)$. 
    The isomorphism $H^1(K,E[2])\to M_K'$ comes from the composition
    \[
    H^1(K,E[2])\xrightarrow{w_K}H^1(K,\mu_2(\overline{M_K}))\xrightarrow{k_K}M_K^{*}/M_K^{*2},
    \]
    and similarly for $L$ (notation of the maps is as in \cite{schaefer1995}, but with subscripts to specify the base field). The desired commutativity thus follows at once because the corestriction maps are natural and compatible with connecting homomorphisms.
\end{proof}

Suppose that $L=K(\sqrt{d})$ is a quadratic extension of $K$, and consider the Weil restriction $A=\mathrm{Res}_{L/K}E_L$. It is an Abelian surface which by definition, see \cite[\S 7.6]{boschluetkebohmertraynaud1980neron}, comes with bijections
\begin{equation}\label{eq:adjunction}
\mathrm{Hom}_L(T_L,E_L)\cong\mathrm{Hom}_K(T,A)
\end{equation}
for any $K$-scheme $T$, natural in $T$. The $G_K$-module $A(\overline{K})$ is induced from the $G_L$-module $E(\overline{L})$, hence by Shapiro's lemma we get $H^n(K,A)\cong H^n(L,E)$ for all $n\geq 0$. For $n=0$ this isomorphism coincides with the isomorphism $E(K)\cong E(L)$ obtained from \eqref{eq:adjunction} with $T=\mathrm{Spec}(K)$. Similarly, $A(\overline{K})[2]$ is induced from $E(\overline{L})[2]$ and we obtain isomorphic long exact sequences
\begin{equation}\label{eq:weilrescohom}
\begin{tikzcd}[column sep=2.9ex, row sep=4ex]
0 \arrow[r] & {A(K)[2]} \arrow[d, "\sim" {rotate=90, anchor=north}] \arrow[r] & A(K) \arrow[d, "\sim" {rotate=90, anchor=north}] \arrow[r] & A(K) \arrow[d, "\sim" {rotate=90, anchor=north}] \arrow[r] & {H^1(K,A[2])} \arrow[d, "\sim" {rotate=90, anchor=north}] \arrow[r] & {H^1(K,A)} \arrow[d, "\sim" {rotate=90, anchor=north}] \arrow[r] & \cdots \\
0 \arrow[r] & {E(L)[2]} \arrow[r]                   & E(L) \arrow[r]           & E(L) \arrow[r]           & {H^1(L,E[2])} \arrow[r]           & {H^1(L,E)} \arrow[r]           & \cdots
\end{tikzcd}
\end{equation}
As $K_{\mathfrak{p}}\otimes L\cong\prod_{\mathfrak{q}|\mathfrak{p}}L_{\mathfrak{q}}$ for a prime $\mathfrak{p}$ of $K$, the isomorphism \eqref{eq:weilrescohom} also induces isomorphisms
\[
\begin{split}
    S^{2}(A/K)&\cong S^{2}(E/L), \\
    \Sh(A/K)&\cong\Sh(E/L).
\end{split}
\]
Now suppose that $E$ and the quadratic twist $E_{d}$ are given by the equations
\begin{alignat*}{2}
    E&:\,\,\,y^2\,&=f(x) \\
    E_{d}&:dy^2\,&=f(x)
\end{alignat*}
for a monic cubic polynomial $f(x)\in K[x]$. The identity on $E_L$ and the isomorphism $E_L\to (E_L)_{d}$ fixing the $x$-coordinate yields by \eqref{eq:adjunction} a map $\phi:E\times E_{d}\to A$ defined over $K$. This is a $2$-isogeny\footnote{For $L/K$ a finite Galois extension, $\mathrm{Res}_{L/K}E_L$ is isogenous to a direct sum of twists of $E$ by the irreducible rational representations of $\mathrm{Gal}(L/K)$, see \cite[Theorem 4.5]{mazurrubinsilverberg2007}.} with kernel equal to the diagonal embedding $\Delta$ of $E_1[2]\cong E_2[2]$, hence $\phi$ factors through $[2]$ as $[2]=\phi^{\vee}\phi$ for a unique $\phi^{\vee}:A\to E\times E_{d}$. As an immediate application, let us quickly give a proof of
\begin{proposition}\label{prop:rankidentitytwist}
We have $\rank(E/L)=\rank(E/K)+\rank(E_{d}/K)$.
\end{proposition}
\begin{proof}
    As the rank is an isogeny invariant we obtain the result as follows.
    \[
    \begin{split}
        \rank(E/L)&=\rank(A/K)\\
        &=\rank(E\times E_{d}/K)\\
        &=\rank(E/K)+\rank(E_{d}/K)\qedhere
    \end{split}
    \]
\end{proof}
In particular, since $\Delta$ is isotropic with respect to the Weil pairing $e_2$ on $(E_1\times E_2)[2]$, we obtain a principal polarisation of $(E\times E_{d})/\Delta\cong A$ compatible with the isogenies $\phi$ and $\phi^{\vee}$. Alternatively, one can use that Weil restriction preserves principal polarisations, see \cite[Prop. 2]{diemnaumann2005}.\

Denote $\phi$ composed with the inclusion $E\to E\times E_{d}$ by $p^*$, and $\phi^{\vee}$ composed with the projection $E\times E_{d}\to E$ by $p_*$ as in \cite[\S 3]{Bruin_2004}. Then we have the following commutative diagram
\[
\begin{tikzcd}
E \arrow[rrdd, swap,"{[2]}"] \arrow[r] \arrow[rr, "p^*", bend left=20] & E\times E_{d} \arrow[rd, "{[2]}"'] \arrow[r, "\phi"'] & A \arrow[d, "\phi^{\vee}"'] \arrow[dd, "p_*", bend left=50] \\
                                                               &                                                     & E\times E_{d} \arrow[d]                               \\
                                                               &                                                     & E                                               
\end{tikzcd}
\]
We see that $[2]$ on $E$ factors through $p^*$ and $p_*$. This also induces a factorisation of $[2]$ on the long exact sequence of cohomology associated to $[2]$ on $E(\overline{K})$, and upon applying the isomorphism of \eqref{eq:weilrescohom} we obtain the following commutative diagram
\[
\begin{tikzcd}[column sep=3.5ex, row sep=4ex]
0 \arrow[r] & {E(K)[2]} \arrow[d] \arrow[r] & E(K) \arrow[d] \arrow[r] & E(K) \arrow[d] \arrow[r] & {H^1(K,E[2])} \arrow[d] \arrow[r] & {H^1(K,E)} \arrow[d] \arrow[r] & \cdots \\
0 \arrow[r] & {E(L)[2]} \arrow[r] \arrow[d] & E(L) \arrow[r] \arrow[d] & E(L) \arrow[r] \arrow[d] & {H^1(L,E[2])} \arrow[r] \arrow[d] & {H^1(L,E)} \arrow[r] \arrow[d] & \cdots \\
0 \arrow[r] & {E(K)[2]} \arrow[r]           & E(K) \arrow[r]           & E(K) \arrow[r]           & {H^1(K,E[2])} \arrow[r]           & {H^1(K,E)} \arrow[r]           & {}    
\end{tikzcd}
\]
where the first row of vertical arrows coincide with restriction maps, and the second row of vertical arrows with corestriction maps. Observe that for $\xi\in\Sh(E/K)[2]$, this implies that if $\xi$ is divisible by $2$ in $\Sh(E/K)[2]$, then $\xi$ has to be in the image of $\Sh(E/L)[2]\xrightarrow{\mathrm{cor}}\Sh(E/K)[2]$. Using Lemma \ref{lem:normcor} we obtain the following commutative diagram.
\begin{equation}\label{eq:mainresdiag}
\begin{tikzcd}
S^{2}(E/L) \arrow[r, two heads] \arrow[d, "N"] & \Sh(E/L)[2] \arrow[d, "\mathrm{cor}"] \\
S^{2}(E/K) \arrow[r, two heads]                & \Sh(E/K)[2]  
\end{tikzcd}
\end{equation}
hence checking if $\xi$ is in the image of $\Sh(E/L)[2]\xrightarrow{\mathrm{cor}}\Sh(E/K)[2]$ can be seen by checking if a lift of $\xi$ in $S^{2}(E/K)$ is a norm from $S^{2}(E/L)$, which can be checked in practice provided that the two relevant $2$-Selmer groups have been computed. This gives a necessary condition for $2$-divisibility in $\Sh(E/K)[2]$, which is used in \cite[Lemma 3.3]{Bruin_2004} to improve on a $2$-descent. It need not be sufficient however, as seen in \cite[\S 7, Example 2]{Bruin_2004}.\\

For an abelian group $M$, let $M_{\mathrm{nd}}=M/M_{\mathrm{div}}$, where $M_{\mathrm{div}}$ is the subgroup of $M$ consisting of the infinitely divisible elements. The Cassels-Tate pairing on $\Sh(A/K)$ (recall that $A$ is principally polarized) yields a non-degenerate pairing
\[
\langle -,-\rangle_A:\Sh(A/K)_{\mathrm{nd}}\times\Sh(A/K)_{\mathrm{nd}}\to\mathbb{Q}/\mathbb{Z},
\]
and similarly for $\Sh(E\times E_{d}/K)$. The pairings $\langle -,-\rangle_A$ and $\langle -,-\rangle_{E\times E_{d}}$ are compatible with respect to the dual isogenies $\phi$ and $\phi^{\vee}$, in the sense that
\[
\langle \phi\xi,\xi'\rangle_{A}=\langle \xi,\phi^{\vee}\xi'\rangle_{E\times E_{d}},
\]
see \cite[Prop. 31]{poonenstoll1999}. This induces a non-degenerate pairing
\begin{equation}\label{eq:ctingredient}
\Sh(A/K)_{nd}[\phi^{\vee}]\times \Sh(A/K)_{nd}/\phi\Sh(E\times E_{d}/K)_{nd}\to\mathbb{Q}/\mathbb{Z}.
\end{equation}
Compare also with \cite[\S 7]{BruinHem}. We can now give a proof of Theorem \ref{thm:mainsha}, whose statement we recall for convenience.
\begin{thmrepeat}
     Suppose that the following conditions are satisfied.
     \begin{enumerate}
         \item We have $\Sh(E_{d}/K)[2]=0$.
         \item The map $\mathrm{cor}:\Sh(E/L)[2]\to\Sh(E/K)[2]$ is injective.
     \end{enumerate}
     Then $\mathrm{res}:\Sh(E/K)\to\Sh(E/L)$ induces an isomorphism
     \[
     \Sh(E/K)[2^{\infty}]/\Sh(E/K)[2]\xrightarrow{\sim}\Sh(E/L)[2^{\infty}].
     \]
\end{thmrepeat}
\begin{proof}
    Using \eqref{eq:weilrescohom} we translate the statements of the theorem to maps only involving Tate-Shafarevich groups over $K$. Specifically, condition (2) translates to the injectivity of $\Sh(A/K)[2]\to\Sh(E/K)[2]$, and we need to show that the map $\Sh(E/K)\to\Sh(A/K)$ induces an isomorphism
    \[
        \Sh(E/K)[2^{\infty}]/\Sh(E/K)[2]\xrightarrow{\sim}\Sh(A/K)[2^{\infty}].
    \]
    The map $p_*$ induces maps $\Sh(A/K)[2]\to\Sh(E/K)[2]$ and $\Sh(A/K)[2^{\infty}]\to\Sh(E/K)[2^{\infty}]$, which we will denote by $p_{*,2}$ and $p_{*,2^{\infty}}$, respectively. Similarly, restricting $\Sh(p^*)$ to the $2$-torsion and $2$-part, respectively, we obtain maps denoted by $p^*_2$ and $p^*_{2^{\infty}}$.\\
    
    The first step is to show that $p_{*,2^{\infty}}$ is injective. Observe that because $p_{*,2}$ is assumed to be injective, this follows directly once we know that $\ker(p_{*,2^{\infty}})\subset\Sh(A/K)[2]$.\\
    As $\Sh$ is additive and $\Sh(E_{d}/K)[2]=0$, the map $\Sh(E\times E_{d}/K)[2^{\infty}]\to \Sh(E/K)[2^{\infty}]$ is an isomorphism. It follows that
    \[
        \ker(p_{*,2^{\infty}})=\Sh(A/K)[\phi^{\vee}]\subset\Sh(A/K)[2],
    \]
    and hence $p_{*,2^{\infty}}$ is injective. Secondly, as $\Sh(A/K)[\phi^{\vee}]=0$ we have
    \[
    \Sh(E\times E_{d}/K)[\phi]=\Sh(E\times E_{d}/K)[2].
    \]
    Since also $\Sh(E/K)[2^{\infty}]\xrightarrow{\sim}\Sh(E\times E_{d}/K)[2^{\infty}]$, we deduce that $\ker(p^*_{2^{\infty}})=\Sh(E/K)[2]$. Thus to complete the proof it suffices to show that $p^*_{2^{\infty}}$ is surjective, which follows from showing that $\Sh(\phi)$ is surjective. As $\Sh(A/K)[\phi^{\vee}]=0$, the non-degeneracy of the pairing \eqref{eq:ctingredient} implies that $\Sh(E\times E_{d}/K)_{\mathrm{nd}}\to\Sh(A/K)_{\mathrm{nd}}$ is surjective. As $\Sh(E\times E_{d}/K)_{\mathrm{div}}\to\Sh(A/K)_{\mathrm{div}}$ is surjective (multiplication by $2$ is clearly surjective on $\Sh(A/K)_{\mathrm{div}}$), we deduce that also $\Sh(\phi)$ itself is surjective, completing the proof.
\end{proof}
Recall the commutative diagram \eqref{eq:mainresdiag}. If $S^{2}(E/K)$ and $S^{2}(E/L)$ have been computed this provides a way to verify condition (2) of Theorem \ref{thm:mainsha}. Before we apply Theorem \ref{thm:mainsha} to prove Theorems \ref{thm:resp} and \ref{thm:respq}, let us discuss the conditions of Theorem \ref{thm:mainsha}. In the proof we saw that the map $\Sh(E/K)\to\Sh(A/K)$ has kernel equal to all of $\Sh(E/K)[2]$. In the language of visualisation, this means that all of $\Sh(E/K)[2]$ is simultaneously visualised in the Abelian surface $A$. From \cite[Cor. 3.4]{Bruin_2004} we know that a $\xi\in\Sh(E/K)[2]$ is in the kernel of $\Sh(E/K)\to\Sh(A/K)$ if and only if $\xi$ is in the image of $E_{d}(K)\to\Sh(E/K)[2]$. Thus we see that a necessary condition for the hypotheses of Theorem \ref{thm:mainsha} to be valid, is that $E_{d}(K)$ hits all of $\Sh(E/K)[2]$.\\

We can see that this indeed happens for the twists corresponding to the quadratic extensions of $\mathbb{Q}$ seen in Propositions \ref{prop:2seloverp} and \ref{prop:2seloverpq}: in the $p\equiv 1,49\bmod 120$ case of Proposition \ref{prop:2selrat} we see that $\Sh(X_0(15)_{-p}/\mathbb{Q})[2]$ is generated by the images of $(15,6,10)$ and $(-1,-1,1)$, which are in the image of $X_0(15)_{-1}(\mathbb{Q})$ as seen in the proof of Proposition \ref{prop:basicrankandshas}! This case is special in the sense that the points in $X_0(15)_{-1}(\mathbb{Q})$ hitting $\Sh(X_0(15)_{-p}/\mathbb{Q})[2]$ are all torsion. In the $p\equiv 17,113\bmod 120$ case such a twist does not exist (a consequence of Proposition \ref{prop:torsiontwists}), but since the explicit generators of $\Sh(X_0(15)_{-p}/\mathbb{Q})[2]$, the images of $(2,2,1)$ and $(-1,-1,1)$, do not depend on $p$, we can take a prime $q$ distinct from $p$ satisfying the same congruences, causing $\Sh(X_0(15)_{-p}/\mathbb{Q})[2]$ to get hit by $X_0(15)_{-q}(\mathbb{Q})$ provided that also $\rank(X_0(15)_{-q})=2$.\\
In fact this also works in the $p\equiv 1,49\bmod 120$ case: one can prove Theorems \ref{thm:mainrank} and \ref{thm:resp} by replacing Proposition \ref{prop:2seloverp} by a variant involving $\mathbb{Q}(\sqrt{pq})$ in the same spirit as Proposition \ref{prop:2seloverpq}. Two possible choices for $q$ in this case (so $q\equiv 1,49\bmod 120$ and with $X_0(15)_{-q}$ of rank $2$) are $q=241$ and $q=601$.
\begin{proposition}\label{prop:shap}
    Let $E=X_0(15)$ and let $p\equiv 1,49\bmod 120$ be a prime. Then we have an isomorphism
    \[
    \Sh(E_{-p}/\mathbb{Q})[2^{\infty}]/\Sh(E_{-p}/\mathbb{Q})[2]\cong \Sh(E_{-1}/\mathbb{Q}(\sqrt{p}))[2^{\infty}]
    \]
\end{proposition}
\begin{proof}
    Let $K=\mathbb{Q}(\sqrt{p})$. We apply Theorem \ref{thm:mainsha} to $E_{-p}$ and $d=p$. Also recall Remark \ref{rem:squareremark}. Since $(E_{-p})_p\cong E_{-1}$ we have $\Sh((E_{-p})_p/\mathbb{Q})[2]=0$ by Proposition \ref{prop:basicrankandshas}. Let us see why $\Sh(E_{-p}/K)[2]\xrightarrow{\mathrm{cor}}\Sh(E_{-p}/\mathbb{Q})[2]$ is injective. When $[-1,10,p]=-1$ this follows immediately as then $\Sh(E_{-p}/K)[2]=0$ by Proposition \ref{prop:2seloverp}, so suppose that $[-1,10,p]=1$. We have the commutative diagram
    \begin{equation}\label{eq:diagshanormp}
    \begin{tikzcd}
        S^{2}(E_{-p}/K) \arrow[r, two heads] \arrow[d, "N"] & \Sh(E_{-p}/K)[2] \arrow[d, "\mathrm{cor}"] \\
        S^{2}(E_{-p}/\mathbb{Q}) \arrow[r, two heads]                & \Sh(E_{-p}/\mathbb{Q})[2]              
    \end{tikzcd}
    \end{equation}
    Note that $\Sh(E_{-p}/K)[2]$ and $\Sh(E_{-p}/\mathbb{Q})[2]$ have the same size: their $\mathbb{F}_2$-dimensions both equal $2-\rank(E_{-p}/\mathbb{Q})$. Thus to prove injectivity it is equivalent to show surjectivity, which follows from \eqref{eq:diagshanormp} and the computation of the 2-Selmer groups: $\Sh(E_{-p}/\mathbb{Q})[2]$ is generated by the images of  $(15,6,10)$ and $(-1,-1,1)$, which are indeed norms: $(15,6,10)$ is the norm of either $(x_3x_5,x_2x_3,x_2x_5)$ or $(x_3x_5,y_2x_3,y_2x_5)$, while $(-1,-1,1)$ is the norm of either $(x_{-1},x_{-1},1)$ or $(3x_{-1},3x_{-1},1)$.
\end{proof}
Similarly we have
\begin{proposition}\label{prop:shapq}
    Let $E=X_0(15)$, and let $p,q\equiv 17,113\bmod 120$ distinct primes such that $\rank(E_{-q}/\mathbb{Q})=2$. Then we have an isomorphism
    \[
        \Sh(E_{-p}/\mathbb{Q})[2^{\infty}]/\Sh(E_{-p}/\mathbb{Q})[2]\cong \Sh(E_{-q}/\mathbb{Q}(\sqrt{pq}))[2^{\infty}]
    \]
\end{proposition}
\begin{proof}
    Let $K=\mathbb{Q}(\sqrt{pq})$. As $\rank(E_{-q}/\mathbb{Q})=2$, Proposition \ref{prop:2selrat} implies that $\Sh(E_{-q}/\mathbb{Q})[2]=0$, thus as in the proof of Proposition \ref{prop:shap}, the desired isomorphism follows from Theorem \ref{thm:mainsha} if we show that the right vertical map in the diagram
    \begin{equation}\label{eq:diagshanormpq}
    \begin{tikzcd}
        S^{2}(E_{-p}/K) \arrow[r, two heads] \arrow[d, "N"] & \Sh(E_{-p}/K)[2] \arrow[d, "\mathrm{cor}"] \\
        S^{2}(E_{-p}/\mathbb{Q}) \arrow[r, two heads]                & \Sh(E_{-p}/\mathbb{Q})[2]              
    \end{tikzcd}
    \end{equation}
    is injective. Note that
    \[
    \begin{split}
    \dim_{\mathbb{F}_2}S^2(E_{-p}/K)-2&=\rank(E_{-q}/K)+\dim_{\mathbb{F}_2}\Sh(E_{-p}/K)[2] \\
    &=\rank(E_{-p}/\mathbb{Q})+2+\dim_{\mathbb{F}_2}\Sh(E_{-p}/K)[2]
    \end{split}
    \]
    As $\rank(E_{-q}/\mathbb{Q})=2$ we must have $[-1,2,q]=-1$ by Theorem \ref{thm:mainrank}, so if $[-1,2,p]=1$ we have $\dim_{\mathbb{F}_2}S^2(E_{-q}/K)=4$ by Proposition \ref{prop:2seloverpq} and hence $\dim_{\mathbb{F}_2}\Sh(E_{-p}/K)[2]=0$, so the injectivity clearly follows.\\
    When $[-1,2,p]=-1$ we have $\dim_{\mathbb{F}_2}S^2(E_{-q}/K)=6$, and we see that $\Sh(E_{-p}/K)[2]$ and $\Sh(E_{-p}/\mathbb{Q})[2]$ both have $\mathbb{F}_2$-dimension $2-\rank(E_{-p}/\mathbb{Q})$, so injectivity is equivalent to surjectivity. As in the proof of Proposition \ref{prop:shap}, this follows from \eqref{eq:diagshanormpq} and the $2$-Selmer group computations. The images of $(1,2,2), (-1,-1,1)\in S^2(E_{-p}/\mathbb{Q})$, generate $\Sh(E_{-p}/\mathbb{Q})[2]$, and they are the norms of $(1,x_2,x_2)$ and $(x_{-1},x_{-1},1)$, respectively.
\end{proof}

Combining the various results we can now prove Theorems \ref{thm:resp} and \ref{thm:respq}. 

\begin{proof}[Proof of Theorem \ref{thm:resp}]
From Proposition \ref{prop:2selrat} we have
\begin{equation}\label{eq:rankshares2}
\rank(E_{-p}/\mathbb{Q})+\dim_{\mathbb{F}_2}\Sh(E_{-p}/\mathbb{Q})[2]=2.
\end{equation}
Suppose that $[-1,10,p]=-1$. Then $\rank(E_{-p}/\mathbb{Q})=0$ by Theorem \ref{thm:mainrank}, hence $\Sh(E_{-p}/\mathbb{Q})[2]\cong (\mathbb{Z}/2\mathbb{Z})^2$ by \eqref{eq:rankshares2}. From Proposition \ref{prop:2seloverp} we have $\Sh(E_{-p}/\mathbb{Q}(\sqrt{p}))[2]=0$, hence from Proposition \ref{prop:shap} we obtain
\[
\Sh(E_{-p}/\mathbb{Q})[2^{\infty}]=\Sh(E_{-p}/\mathbb{Q})[2],
\]
which proves \ref{thm:resp:i}.
Now assume $[-1,10,p]=1$. Then multiplication by $2$ induces an injection
\begin{equation}\label{eq:2surjp}
\Sh(E_{-p}/\mathbb{Q})[4]/\Sh(E_{-p}/\mathbb{Q})[2]\to\Sh(E_{-p}/\mathbb{Q})[2].
\end{equation}
The left hand side of \eqref{eq:2surjp} is isomorphic to $\Sh(E_{-p}/\mathbb{Q}(\sqrt{p}))[2]$ by Proposition \ref{prop:shap}, which has the same size as $\Sh(E_{-p}/\mathbb{Q})[2]$. It follows that the map \eqref{eq:2surjp} is also surjective, which proves \ref{thm:resp:ii}.\\
If we additionally assume that $\rank(E_{-p}/\mathbb{Q})=0$, then $\Sh(E_{-p}/\mathbb{Q})[2]\cong(\mathbb{Z}/2\mathbb{Z})^2$, which combined with surjectivity of $\Sh(E_{-p}/\mathbb{Q})[4]\xrightarrow{\cdot 2}\Sh(E_{-p}/\mathbb{Q})[2]$ forces
\[
\Sh(E_{-p}/\mathbb{Q})[4]\cong (\mathbb{Z}/4\mathbb{Z})^2,
\]
proving \ref{thm:resp:iii}.
\end{proof}

\begin{proof}[Proof of Theorem \ref{thm:respq}]
Again we have
\[
\rank(E_{-p}/\mathbb{Q})+\dim_{\mathbb{F}_2}\Sh(E_{-p}/\mathbb{Q})[2]=2.
\]
When $p=17$ we have $[-1,2,p]=-1$, $\rank(E_{-p}/\mathbb{Q})=2$ and $\Sh(E_{-p}/\mathbb{Q})[2]=0$, so the theorem clearly holds in this case. Thus we may assume $p\neq 17$ for the rest of the proof. Let $K=\mathbb{Q}(\sqrt{17p})$. By Proposition \ref{prop:2seloverpq} and \eqref{eq:rankformulatwist} we obtain
\[
    \rank(E_{-p}/\mathbb{Q})+\dim_{\mathbb{F}_2}\Sh(E_{-17}/K)[2]=
    \begin{cases*}
        0 & if\, $[-1,2,p]=1$ \\
        2 & if\, $[-1,2,p]=-1$
    \end{cases*}
\]

If $[-1,2,p]=1$ then $\rank(E_{-p}/\mathbb{Q})=0$, and with Proposition \ref{prop:shapq} we have
    \[
        \Sh(E_{-p}/\mathbb{Q})[2^{\infty}]=\Sh(E_{-p}/\mathbb{Q})[2]\cong(\mathbb{Z}/2\mathbb{Z})^2,
    \]
which proves \ref{thm:respq:i}. If $[-1,2,p]=-1$ then multiplication by $2$ gives the injection
\[
\Sh(E_{-p}/\mathbb{Q})[4]/\Sh(E_{-p}/\mathbb{Q})[2]\to\Sh(E_{-p}/\mathbb{Q})[2],
\]
which is surjective as the left hand side is isomorphic to $\Sh(E_{-p}/K)[2]$ by Proposition \ref{prop:shapq}, which has the same size $\Sh(E_{-p}/\mathbb{Q})[2]$. This proves \ref{thm:respq:ii}, and \ref{thm:respq:iii} follows just as in the proof of Theorem \ref{thm:resp}.
\end{proof}

\section{Concluding remarks}\label{sec:conclusionsection}

\noindent Theorems \ref{thm:resp} and \ref{thm:respq} show that the upper bounds obtained on $\rank(X_0(15)/\mathbb{Q}(\sqrt{-p}))$ in Theorem \ref{thm:mainrank} are the same as those obtained from performing a $4$-descent over $\mathbb{Q}$ on the twist $X_0(15)_{-p}$. In fact Theorems \ref{thm:resp} and \ref{thm:respq} can also be proven by a computation of the Cassels-Tate pairing on $S^2(X_0(15)_{-p}/\mathbb{Q})$, as explained in \cite{Cassels1998}. The advantage of the proof as presented in this paper is that Propositions \ref{prop:shap} and \ref{prop:shapq} show that the rank bound obtained from $2^n$-descent over $\mathbb{Q}$ on $X_0(15)_{-p}$ is equivalent to the one obtained from a $2^{n-1}$-descent on $X_0(15)_{-1}$ or $X_0(15)_{-17}$ over a suitable quadratic field, opening up possibilities to further improve on the rank bounds in various ways.

For example, a Cassels-Tate pairing computation on $S^2(X_0(15)_{-1}/\mathbb{Q}(\sqrt{p}))$ for primes $p\equiv 1,49\bmod 120$ gives the same rank bounds as an $8$-descent on $X_0(15)_{-p}$ over $\mathbb{Q}$. Consider for instance the prime $p=1609$. Then $p\equiv 49\bmod 120$ and $[-1,10,p]=1$, and with Magma one computes that the Cassels-Tate pairing on $S^2(X_0(15)_{-1}/\mathbb{Q}(\sqrt{p}))$ is non-trivial, which proves that
\[
\rank(X_0(15)_{-p}/\mathbb{Q})=0,\quad\text{and }\quad\Sh(X_0(15)_{-p}/\mathbb{Q})[2^{\infty}]\cong (\mathbb{Z}/4\mathbb{Z})^2.
\]
Alternatively, one can work with explicit equations for the $2$-covers corresponding to the elements of $S^2(X_0(15)_{-1}/\mathbb{Q}(\sqrt{p}))$ and perform a second descent as in \cite[\S 9]{bruinstoll2009}.\\

Let us discuss the rank of $X_0(15)$ over imaginary quadratic fields of the form $\mathbb{Q}(\sqrt{-dp})$, where $d$ is positive, square-free and not divisible by primes $>5$, and $p>5$ is a prime. A $2$-Selmer group computation over $\mathbb{Q}$ of $X_0(15)_{-dp}$ gives an upper bound of either $0$ or $2$ on the rank provided that $dp\equiv 0,1,2,3,4,5,8,12\bmod 15$. In cases where the rank bound is $2$, depending on the class of $p\bmod 120$, one can find $2$ independent elements in the $2$-Selmer group that are independent of $p$, just as in in Proposition \ref{prop:2selrat}.
Experimentation with Magma seems to suggest that Theorem \ref{thm:mainsha} is applicable and yields rank bounds equivalent to those obtained from $4$-descent on $X_0(15)_{-dp}$ by finding suitable quadratic twists satisfying the necessary condition as discussed after the proof of Theorem \ref{thm:mainsha} in section \ref{sec:shasection}.\\

Theorem \ref{thm:mainsha} is also applicable to the quadratic twists $E_p$ of $E:y^2=x^3-x$, where $p\equiv 1 \bmod 8$ is a prime as one encounters in studying congruent primes. Specifically, the rank bounds obtained from $2$-Selmer group computations over $\mathbb{Q}(\sqrt{p})$ in \cite{evinktoptop2022} can be shown to be equivalent to those from a $4$-descent of $E_p$ over $\mathbb{Q}$ by an application of Theorem \ref{thm:mainsha}. If one goes one step further with for example a Cassels-Tate pairing computation over $\mathbb{Q}(\sqrt{p})$ one arrives at a proof of a variant of \cite[Theorem B]{BruinHem} as mentioned in the last sentence of section 10 of loc. cit. The variance lies in the fact that the number theoretic description is different, even though they can likely be proven to be equivalent directly through a reciprocity argument, similar to how the $2$-descent of $E$ over $\mathbb{Q}(\sqrt{p})$ is expressed using $[p,-1,2]$, and the $2$-descent of the isogeneous curve is expressed using $[-1,2,p]$, which are the same because of R\'edei reciprocity.\\

Let us conclude by discussing the restrictiveness of the hypotheses of Theorem \ref{thm:mainsha} in general. As remarked before, if the hypotheses are fulfilled, this implies that $\Sh(E/K)\to\Sh(A/K)$ has kernel equal to $\Sh(E/K)[2]$, forcing surjectivity of $E_{d}(K)\to \Sh(E/K)[2]$. In particular this implies
\[
\rank(E_{d}/K)\geq\mathrm{dim}_{\mathbb{F}_2}\Sh(E/K)[2]-\mathrm{dim}_{\mathbb{F}_2}E(K)[2],
\]
which, as $\Sh(E/K)[2]$ increases in size, becomes increasingly restrictive. Specialising to $K=\mathbb{Q}$, then the size of $\Sh(E/\mathbb{Q})[2]$ is known to be unbounded as $E$ varies. Assuming that the ranks of elliptic curves over $\mathbb{Q}$ are uniformly bounded, this implies there exist elliptic curves over $\mathbb{Q}$ for which the hypotheses of Theorem \ref{thm:mainsha} are not fulfilled for any quadratic twist.\\

\section*{Acknowledgements}
I would like to thank Stevan Gajovic for informing me about the modularity result of Newton and Caraiani, which initiated the process of writing this paper. Gratitude goes to Michael Stoll, who directed me to the visualisation techniques, which opened the door to investigate the connection of the rank results to a $4$-descent. I also thank Nils Bruin for fruitful discussions regarding his paper \cite{Bruin_2004} during his visit to Ulm, and for helpful clarifications of some later questions. Finally, Irene Bouw and Jeroen Sijsling gave valuable suggestions to preliminary versions of this paper.

\printbibliography

@article {magma,
    AUTHOR = {Bosma, Wieb and Cannon, John and Playoust, Catherine},
     TITLE = {The {M}agma algebra system. {I}. {T}he user language},
      NOTE = {Computational algebra and number theory (London, 1993)},
   JOURNAL = {J. Symbolic Comput.},
    VOLUME = {24},
      YEAR = {1997},
    NUMBER = {3-4},
     PAGES = {235--265}
}

@article {bruinstoll2009,
    AUTHOR = {Bruin, Nils and Stoll, Michael},
     TITLE = {Two-cover descent on hyperelliptic curves},
   JOURNAL = {Math. Comp.},
    VOLUME = {78},
      YEAR = {2009},
    NUMBER = {268},
     PAGES = {2347--2370},
}

@article {poonenstoll1999,
    AUTHOR = {Poonen, Bjorn and Stoll, Michael},
     TITLE = {The {C}assels-{T}ate pairing on polarized abelian varieties},
   JOURNAL = {Ann. of Math.},
    VOLUME = {150},
      YEAR = {1999},
    NUMBER = {3},
     PAGES = {1109--1149},
}

@article{stoll2001,
author = {Michael Stoll},
journal = {Acta Arithmetica},
language = {eng},
number = {3},
pages = {245-277},
title = {Implementing 2-descent for Jacobians of hyperelliptic curves},
url = {http://eudml.org/doc/279762},
volume = {98},
year = {2001},
}

@article{Cassels1998,
author = {Cassels, John William Scott},
url = {https://doi.org/10.1515/crll.1998.001},
title = {Second descents for elliptic curves},
title = {},
pages = {101--127},
volume = {1998},
number = {494},
journal = {Journal für die reine und angewandte Mathematik},
doi = {doi:10.1515/crll.1998.001},
year = {1998},
lastchecked = {2024-03-15}
}

@article{schaefer1995,
  title={2-descent on the Jacobians of hyperelliptic curves},
  author={Schaefer, Edward Frank},
  journal={Journal of number theory},
  volume={51},
  number={2},
  pages={219--232},
  year={1995},
  publisher={New York, Academic Press, 1969-}
}

@article{mazurrubinsilverberg2007,
  title={Twisting Commutative Algebraic Groups},
  author={Mazur, Barry and Rubin, Karl and Silverberg, Alice},
  journal={Journal of Algebra},
  year={2006},
  volume={314},
  pages={419-438}
}

@book{serre1998,
  title={Abelian L-Adic Representations and Elliptic Curves},
  author={Serre, Jean-Pierre},
  isbn={9781568810775},
  series={Research Notes in Mathematics (A K Peters), Vol 7},
  year={1998},
  publisher={Peters}
}

@book{bourbaki1998algebra,
  title={Algebra I: Chapters 1-3},
  author={Bourbaki, Nicolas},
  isbn={9783540642435},
  lccn={88031211},
  series={Actualit{\'e}s scientifiques et industrielles},
  year={1998},
  publisher={Springer}
}

@book{boschluetkebohmertraynaud1980neron,
  title={N{\'e}ron Models},
  author={Bosch, Siegfried and L{\"u}tkebohmert, Werner and Raynaud, Michel},
  isbn={0-387-50587-3},
  year={1980},
  publisher={Springer}
}

@book{silverman2000,
  title={The Arithmetic of Elliptic Curves},
  author={Joseph H. Silverman},
  isbn={978-0-387-09493-9},
  year={2000},
  publisher={Springer},
  edition={2nd Edition}
}

@article{stevenhagen2021redei,
    title={Redei reciprocity, governing fields, and negative Pell},
    author={Peter Stevenhagen},
    year={2021},
    pages={1--28},
    journal={Math. {P}roc. {C}ambridge {P}hilos. {S}oc.}, 
}

@misc{caraianinewton2023modularity,
  doi = {10.48550/ARXIV.2301.10509},
  
  url = {https://arxiv.org/abs/2301.10509},
  
  author = {Caraiani, Ana and Newton, James},
  
  title = {On the modularity of elliptic curves over imaginary quadratic fields},
  
  publisher = {arXiv},
  
  year = {2023},
  
  copyright = {arXiv.org perpetual, non-exclusive license}
}

@article{Bruin_2004,
   title={Visualising Sha[2] in Abelian surfaces},
   volume={73},
   DOI={10.1090/s0025-5718-04-01633-3},
   journal={Mathematics of Computation},
   publisher={American Mathematical Society (AMS)},
   author={Bruin, Nils},
   year={2004},
   pages={1459–1477}
}

@article {BruinHem,
    AUTHOR = {Bruin, Nils and Hemenway, Brett},
     TITLE = {On congruent primes and class numbers of imaginary quadratic
              fields},
   JOURNAL = {Acta Arith.},
    VOLUME = {159},
      YEAR = {2013},
    NUMBER = {1},
     PAGES = {63--87},
}

@misc{lmfdb,
  author       = {The {LMFDB Collaboration}},
  title        = {The {L}-functions and modular forms database},
  howpublished = {\url{https://www.lmfdb.org}},
  year         = {2024},
  note         = {[Online; accessed 11 March 2024]},
}

@book{milne2006,
author={Milne, James},
title={Arithmetic Duality Theorems},
year={2006},
publisher={BookSurge, LLC},
edition={Second},
pages={viii+339},
isbn={1-4196-4274-X}
}

@article{diemnaumann2005,
author = {Claus, Niem and Niko Naumann},
title = {On the Structure of the Weil Restriction of Abelian Varieties},
journal = {J. Ramanujan Math. Soc.},
volume = {18},
number = {2},
pages = {153-174}
}

@article{evinktoptop2022,
author = {Evink, Tim and Top, Jaap and Top, Jakob Dirk},
title = {A remark on prime (non)congruent numbers},
journal = {Quaestiones Mathematicae},
volume = {45},
number = {12},
pages = {1841-1853},
year = {2022},
publisher = {Taylor & Francis},
doi = {10.2989/16073606.2021.1977410},
}

@article{evinkheidentop2021,
author = {Evink, Tim and van der Heiden, Gert-Jan and Top, Jaap},
title = {Two-descent on some genus two curves},
journal = {Indagationes Mathematicae},
volume = {32},
number = {4},
pages = {883-900},
year = {2021},
issn = {0019-3577},
doi = {https://doi.org/10.1016/j.indag.2021.06.004},
}

@article{vdokchitser2005,
    author = {Dokchitser, Vladimir},
    title = {Root Numbers of Non-Abelian Twists of Elliptic Curves},
    journal = {Proceedings of the London Mathematical Society},
    volume = {91},
    number = {2},
    pages = {300-324},
    year = {2005},
    month = {09},
}

@Inbook{tdokchitser2013,
author={Dokchitser, Tim},
title={Notes on the Parity Conjecture},
bookTitle={Elliptic Curves, Hilbert Modular Forms and Galois Deformations},
year={2013},
publisher={Springer Basel},
address={Basel},
pages={201--249},
}
\end{document}